\documentclass[10pt]{amsart}
\usepackage{latexsym,enumerate}
\usepackage{amsmath,amsthm,amsopn,amstext,amscd,amsfonts,amssymb}
\usepackage{color}
\usepackage{cancel}
\usepackage{graphicx}
\usepackage[active]{srcltx}
\newcommand{\re}{{I\!\!R}}
\newcommand{\ren}{\re^N}

\newcommand{\dyle}{\displaystyle}

\newcommand{\dint}{\dyle\int}
\newcommand{\io}{\int\limits_{\O}}

\newcommand{\Div}{\text{div\,}}

\renewcommand{\a }{\alpha }
\renewcommand{\b }{\beta }

\newcommand{\D }{\Delta }

\newcommand{\e }{\varepsilon }

\renewcommand{\l }{\lambda }
\renewcommand{\L }{\Lambda }

\newcommand{\n }{\nabla }

\newcommand{\s }{\sigma }

\renewcommand{\O }{\Omega }

\newcommand{\cqd}{{\unskip\nobreak\hfil\penalty50
        \hskip2em\hbox{}\nobreak\hfil\mbox{\rule{1ex}{1ex} \qquad}
        \parfillskip=0pt \finalhyphendemerits=0\par\medskip}}
\newenvironment{pf}{\noindent{\sc Proof}.\enspace}{\rule{2mm}{2mm}\medskip}
\newtheorem{Theorem}{Theorem}[section]

\newtheorem{Definition}[Theorem]{Definition}
\newtheorem{Lemma}[Theorem]{Lemma}

\newtheorem{remarks}[Theorem]{Remarks}
\newtheorem{remark}[Theorem]{Remark}
\begin{document}
\title[The Fractional p-laplacian equations]{On  the Fractional p-laplacian equations with weight and general datum.}
\thanks{Work partially supported by Project MTM2013--40846-P, MINECO, Spain}
\author[B. Abdellaoui, A. Attar  \& R. Bentifour]{B. Abdellaoui, A. Attar  \& R. Bentifour}

\address{\hbox{\parbox{5.7in}{\medskip\noindent {Laboratoire d'Analyse Nonlin\'eaire et Math\'ematiques
Appliqu\'ees. \hfill \break\indent D\'epartement de
Math\'ematiques, Universit\'e Abou Bakr Belka\"{\i}d, Tlemcen,
\hfill\break\indent Tlemcen 13000, Algeria.\\[3pt]
        \em{E-mail addresses: }{\tt boumediene.abdellaoui@inv.uam.es, \tt ahm.attar@yahoo.fr, \tt rachidbentifour@gmail.com}.}}}}


\begin{abstract}
The aim of this paper is to study the following problem
$$
(P)
\left\{
\begin{array}{rcll}
(-\Delta)^s_{p, \beta} u &= & f(x,u) &\mbox{ in }\O,\\
 u & = & 0  \    \ &\mbox{ in }  \ren\setminus\O,
\end{array}
\right.
$$
where $\Omega$ is a smooth bounded domain of $\ren$ containing the origin,
 $$ (-\Delta)^s_{p,\beta}\, u(x):=P.V. \int_{\ren} \,\dfrac{|u(x)-u(y)|^{p-2}(u(x)-u(y))}{|x-y|^{N+ps}} \dfrac{dy}{|x|^\b|y|^\b},$$
with $0\le \beta<\frac{N-ps}{2} $, $1<p<N$, $s\in (0,1)$ and $ps<N$.

The main purpose of this work is to prove the existence of a weak solution under some hypotheses on $f$. In particular, we will consider two cases:
\begin{enumerate}
\item $f(x,\s)=f(x)$, in this case we prove the existence of a weak solution, that is in a suitable weighted fractional Sobolev spaces for all $f\in L^1(\Omega)$. In addition, if $f\gneq 0$, we show that problem $(P)$ has a unique entropy positive solution.
\item $f(x,\s)=\l \s^q +g(x),\: \s\ge 0$, in this case, according to the values of $\l$ and $q$, we get the largest class of data $g$ for which problem $(P)$ has a
    positive solution.
    \end{enumerate}

In the case where $f\gneq 0$, then the solution $u$ satisfies a suitable weak Harnack inequality.

\end{abstract}

\maketitle

\section{Introduction and motivations}\label{sec:s0}
We consider the following problem
\begin{equation}\label{eq:def}
\left\{
\begin{array}{rcll}
(-\Delta)^s_{p, \beta} u &= &f(x,u) & \mbox{ in }\O,\\
u &=& 0  & \mbox{ in }  \ren\setminus\O,
\end{array}
\right.
\end{equation}
where
$$ (-\Delta)^s_{p,\beta}\, u(x):=P.V. \int_{\ren} \,\dfrac{|u(x)-u(y)|^{p-2}(u(x)-u(y))}{|x-y|^{N+ps}} \dfrac{dy}{|x|^\b|y|^\b},$$
$\Omega$ is a smooth bounded domain containing the origin and $f$ belongs to a suitable Lebesgue space.

This class of operators appear in a natural way when dealing with the improved Hardy inequality, namely, for all $\phi\in \mathcal{C}^\infty_0(\Omega)$, we have
\begin{equation}\label{sara0}
G_{s,p}(\phi)\geq C
\int_{\Omega}\int_{\Omega}\frac{|v(x)-v(y)|^p}{|x-y|^{N+qs}}w(x)^{\frac{p}{2}}w(y)^{\frac{p}{2}}\,dx\,dy,
\end{equation}
where
$$
G_{s,p}(\phi)\equiv \dint_{\re^N}\dint_{\re^N}
\dfrac{|\phi(x)-\phi(y)|^{p}}{|x-y|^{N+ps}}dx dy -\Lambda_{N,p,s}
\dint_{\re^N} \dfrac{|\phi(x)|^p}{|x|^{ps}} dx,
$$
$\L_{N,p,s}$ is the optimal Hardy constant, $w(x)=|x|^{-\frac{N-ps}{p}}$ and $v(x)=\dfrac{\phi(x)}{w(x)}$. We refer
\cite{FLS}, \cite{APP} and \cite{AB} for a complete discussion about this fact.

In the same way, we can consider $(-\Delta)^s_{p, \beta}$ as an extension of the local operator $-\text{div}(|x|^{-\beta}|\n u|^{p-2}\n u)$. This last one
is strongly related to the classical Caffarelli-Khon-Nirenberg inequalities given in \cite{CKN} and it was deeply analyzed in the literature.
Notice that, as a consequence of the Caffarelli-Khon-Nirenberg inequalities, it is known that the  weight $|x|^{-\beta}$,
with $\b<N-p$, is an admissible weight in the sense that, if $u$ is a weak positive supersolution to problem
$$
-\Div(|x|^{-\beta}|\n u|^{p-2}\n u)=0,
$$
then it satisfies a weak Harnack inequality.

More precisely, there exists a positive constant $\kappa>1$ such
that for all $0<q<\kappa(p-1)$,
$$
\Big(\int_{B_{2\rho}(x_0)}u^q(x)|x|^{-p\beta} dx\Big)^{\frac 1q} \le
C\inf_{B_{\rho}(x_0)}u, $$ where $B_{2\rho}(x_0)\subset\subset\O$,
and $C>0$ depends only on $B$.

We refer to \cite{FB}, \cite{JTO} and the references therein for a complete discussion and the proof of the Harnack inequality and a generalization of admissible weights.

\

Our objective in this work is to analyze the properties of the operator $(-\Delta)^s_{p, \beta}$  and to get the existence of a solution, in a suitable sense, to problem \eqref{eq:def} for the largest class of the datum $f$.

\

  The case of $p-$laplacian equation is well known in the literature, we refer for example to \cite {BG} and \cite{BZ} where the authors proved the existence and the uniqueness of entropy solution for $L^1$ datum. The case of measure datum was treated in \cite{DM}, the existence of a renormalized solution is obtained.

The local case with weight was considered in \cite{AP0}. The authors proved the existence and the uniqueness of entropy solution for datum in $L^1$.

For the operator $(-\Delta)^s_{p,\beta}$, the case $p=2$ and $\beta=0$ was analyzed in \cite{LPPS} and \cite{KPU}. Using a duality argument, in the sense of Stampacchia, the authors were able to prove the existence of solution for any datum in $L^1$. A more general semilinear problem was considered in \cite{ABB}, where the existence and the uniqueness of the solution is studied.

The case $p\neq 2$ and $\beta=0$ was recently treated in \cite{KMS}. Based on some generalization of the Wolf potential theory, the authors succeeded to obtain the existence of a weak solution belonging to a suitable fractional Sobolev space.

In this paper we will treat the case $p\neq 2$ and $\beta>0$. The argument considered in \cite{KMS} seems to be complicated to be adapted to our case.

Our approach is more simple and it is based on a suitable choice of a test function's family and on some algebraic inequalities.

In the first part of the present paper, we will consider the case $f(x,\s)=f(x)$. We prove the existence of a weak solution that is in an appropriate fractional Sobolev space. More precisely we get the next existence result.
\begin{Theorem}\label{mainth}
Assume that $f\in L^1(\O)$, then problem \eqref{eq:def} has a weak solution $u$ such that
\begin{equation}\label{estimm}
\io\io\dfrac{|u(x)-u(y)|^{q}}{|x-y|^{N+qs_1}}\dfrac{1}{|x|^\b|y|^\b}dy
\ dx\le M\mbox{   for all
}q<\dfrac{N(p-1)}{N-s}, \mbox{  and for all
}s_1<s,
\end{equation}
and $T_k(u)\in W^{s,p}_{\beta, 0}(\O)$, for all $k>0$, where
\begin{gather*}\label{f-trun00}
T_k(a)=\left\{\begin{array}{cl}
a\,,&\hbox{ if }|a|\le k\,;\\[2mm]
k\frac{a}{|a|}\,,&\hbox{ if }|a|> k.
\end{array}\right.
\end{gather*}
If $p>2-\dfrac{s}{N}$, then $u\in W^{s_1,q}_{\beta, 0}(\O)$ for all $1\le q<\dfrac{N(p-1)}{N-s}, \mbox{  and for all
}s_1<s$.
\end{Theorem}
It is clear that for $\beta=0$, we reach the same existence and regularity result obtained in \cite{KMS}, however, it seems that our approach is more simple and can be adapted for a large class of weighted nonlocal operators.

\

Next, assuming that $f\ge 0$, we show the existence of positive entropy solution in the sense of Definition \ref{def:entropy}. The statement of our result is the following.
\begin{Theorem}\label{entropi}
Assume that $f\in L^1(\O)$ is such that $f\gneq 0$, then problem \eqref{eq:def}  has a unique entropy positive solution $u$ in the sense of Definition \ref{def:entropy} given below. Moreover if $u_n$ is the unique solution to the approximating problem
\begin{equation}\label{proOO}
\left\{\begin{array}{rcll}
(-\Delta)^s_{p, \beta}u_n &= & f_n(x) & \mbox{  in  }\O,\\
u_n &= & 0 & \mbox{ in } \ren\backslash\O,
\end{array}
\right.
\end{equation}
with $f_n=T_n(f)$, then $T_k(u_n)\to T_k(u)$ strongly in $W^{s,p}_{\beta, 0}(\Omega)$.
\end{Theorem}

In the second part of the paper we consider the case $f(x,\s)=\l \s^q+f(x)$. According to the values of $q$ and $\l$, we prove the existence
of entropy solution for the largest class of the datum $f$.

Finally, for positive datum, we will show that the operator $(-\Delta)^s_{p,\beta}$ satisfies a suitable local Harnack inequality. This last result was proved in \cite{CKP} for $\beta=0$ and in \cite{AMPP2} for the case $p=2$ and $\beta>0$. Here combining the technics of the above papers, we will show the result for the nonlinear case $p\neq 2$ and $\beta>0$.

The paper is organized as follows. In Section  \ref{sec1} we
introduce some useful tools and preliminaries that we will use through the paper, like the weighted fractional Sobolev spaces and some related inequalities, a weak comparison principle and some algebraic inequalities. We also precise the sense in which the solutions to problem \eqref{eq:def} are defined.

In Section \ref{sec2}, we begin by proving Theorem \ref{mainth}, namely, the case where $f(x,\s)\equiv f(x)$. The main idea is to proceed by approximation and to pass to the limit using suitable test functions. In the second part of the section we prove Theorem \ref{entropi}, more precisely, if $f\ge 0$, we are able to show that the problem \eqref{eq:def} has {\bf a unique positive entropy solution}. In the same way, setting $u_n$ the solution of \eqref{eq:def} with datum $f_n\equiv T_n(f)$,  we will prove that the sequence $\{T_k(u_n)\}_n$ converges to $T_k(u)$ strongly in the corresponding weighted fractional Sobolev space.

In Section \ref{sec3}, we study the case where $f(x,\s)=\l \s^q+g(x)$, with $\l>0$ and $g\gneq 0$. According to the values of $q$ and $\l$, we get the largest class of the data $g$ such that the problem \eqref{eq:def} has a positive solution.

In the appendix, and following the argument used in
\cite{CKP} and \cite{AMPP2}, when the datum is positive, we are able to prove a weak version of the Harnack inequality for the operator $(-\Delta)^s_{p,\beta}$.

\section{Functional setting and main tools}\label{sec1}
In this section we give some functional settings that will be used
below. We refer to \cite{DPV} and \cite{MAZ} for more details.

Let $s\in (0,1)$, $p\ge 1$ and $0\le \beta<\frac{N-ps}{2}$. For simplicity of typing, we will set $$d\mu:=
\dfrac{dx}{|x|^{2\beta}} \quad \hbox{ and } \quad d\nu:=
\dfrac{dxdy}{|x-y|^{N+ps}|x|^\beta|y|^\beta}.$$

Let $\O\subset \ren$, the weighted fractional Sobolev
space $W^{s,p}_\beta(\Omega)$ is defined by
$$
W^{s,p}_\beta(\Omega)\equiv
\Big\{ \phi\in
L^p(\O,d\mu):\dint_{\O}\dint_{\O}|\phi(x)-\phi(y)|^pd\nu<+\infty\Big\}.
$$
$W^{s,p}_\b (\O)$ is a Banach space endowed with the norm
$$
\|\phi\|_{W^{s,p}_\b (\O)}=
\Big(\dint_{\O}|\phi(x)|^pd\mu\Big)^{\frac 1p}
+\Big(\dint_{\O}\dint_{\O}|\phi(x)-\phi(y)|^pd\nu\Big)^{\frac
1p}.
$$
In the same way we define the space $W^{s,p}_{\b,0} (\O)$ as
the completion of $\mathcal{C}^\infty_0(\O)$ with respect to the
previous norm.

As in \cite{Adams}, see also \cite{DPV}, we can prove the
following extension result.
\begin{Lemma}\label{ext}
Assume that $\Omega\subset \ren$ is a regular domain, then for all
$w\in W^{s,p}_\b (\O)$, there exists $\tilde{w}\in
W^{s,p}_\b (\ren)$ such that $\tilde{w}_{|\Omega}=w$ and
$$
||\tilde{w}||_{W^{s,p}_\b (\ren)}\le C
||w||_{W^{s,p}_\b (\O)},
$$
where $C\equiv C(N,s,p,\O)>0$.
\end{Lemma}

The following weighted Sobolev inequality is obtained in \cite{AB}
and will be used systematically in this paper.
\begin{Theorem} \label{Sobolev}(Weighted fractional Sobolev inequality)
Assume that $0<s<1$ and $p>1$ are such that $ps<N$. Let
$\beta<\dfrac{N-ps}{2}$, then there exists a positive constant $S(N,s,\beta)$ such that for all
$v\in C_{0}^{\infty}(\ren)$,
$$
\dint_{\mathbb{R}^{N}}\dint_{\mathbb{R}^{N}}
\dfrac{|v(x)-v(y)|^{p}}{|x-y|^{N+ps}}\frac{dx}{|x|^{\beta}}\,
\frac{dy}{|y|^{\beta}}\geq S(N,s,\beta)
\Big(\dint_{\mathbb{R}^{N}}
\dfrac{|v(x)|^{p_{s}^{*}}}{|x|^{2\beta\frac{p_{s}^{*}}{p}}}\Big)^{\frac{p}{p^{*}_{s}}},
$$
where $p^{*}_{s}= \dfrac{pN}{N-ps}$.

Moreover, if $\Omega \subset \ren$ is a bounded domain, and
$\beta=\dfrac{N-ps}{2}$, then for all $q<p$, there exists a
positive constant $C(\Omega)$ such that
$$
\dint_{\mathbb{R}^{N}}\dint_{\mathbb{R}^{N}}
\dfrac{|v(x)-v(y)|^{p}}{|x-y|^{N+ps}}\frac{dx}{|x|^{\beta}}\,
\frac{dy}{|y|^{\beta}}\geq C(\Omega) \Big(\dint_{\mathbb{R}^{N}}
\dfrac{|v(x)|^{p_{s,q}^{*}}}{|x|^{2\beta\frac{p_{s,q}^{*}}{p}}}\Big)^{\frac{p}{p_{s,q}^{*}}},
$$
for all $v\in C_{0}^{\infty}(\Omega)$, where $p^{*}_{s,q}=
\dfrac{pN}{N-qs}$.
\end{Theorem}

\begin{remark}\label{equiv}
As in the case $\beta=0$, if $\O$ is a bounded smooth domain of $\ren$, we
can endow $W^{s,p}_{\beta, 0}(\O)$ with the equivalent norm
$$
|||\phi|||_{W^{s,p}_{\beta, 0}(\O)}=
\Big(\dint_{\O}\dint_{\O}\dfrac{|\phi(x)-\phi(y)|^p}{|x-y|^{N+ps}}\dfrac{dxdy}{|x|^\beta|y|^\beta}\Big)^{\frac
1p}.
$$
\end{remark}
Now, for $w\in W^{s,p}_{\beta}(\ren)$, we set
$$
(-\Delta)^s_{p, \beta} w(x)=\mbox{ P.V. }
\dint_{\ren}\dfrac{|w(x)-w(y)|^{p-2}(w(x)-w(y))}{|x-y|^{N+ps}}\dfrac{dy}{|x|^\beta|y|^\beta}.
$$
It is clear that for all $w, v\in W^{s,p}_{\beta}(\ren)$, we have
$$
\langle (-\Delta)^s_{p, \beta}w,v\rangle
=\dfrac 12\dint_{\ren}\dint_{\ren}\dfrac{|w(x)-w(y)|^{p-2}(w(x)-w(y))(v(x)-v(y))}{|x-y|^{N+ps}}\dfrac{dxdy}{|x|^\beta|y|^\beta}.
$$
In the case where $\beta=0$, we denote $(-\Delta)^s_{p, \beta}$ by
$(-\Delta)^s_{p}$.

The next Picone's inequality is obtained in \cite{LPPS} and \cite{AB}.
\begin{Theorem}{\it (Picone's type Inequality).}\label{Picone}
Let $w\in W^{s,p}_{\beta, 0}(\O)$ be such that $w>0$ in $\O$, and assume that
$(-\Delta)^s_{p, \beta}(w)\ge 0$. Then for all $v\in \mathcal{C}^\infty_0(\O)$ we have
\begin{equation}\label{picone1}
\frac {1}{2}
 \iint_{D_\O} |v(x)-v(y)|^pd\nu\ge
\langle (-\Delta)^s_{p, \beta}w,\frac{|v|^p}{w^{p-1}}\rangle,
\end{equation}
where $D_{\O}=(\ren\times\ren)\backslash (C\O\times C\O) $.
\end{Theorem}
As a consequence the next comparison principle is obtained, that
extends the classical one obtained by Brezis and Kamin in \cite{BK}. See \cite{AB} for the proof.

\begin{Lemma}\label{compa}
Let $\O$ be  a bounded domain and let $h$ be a non-negative
continuous function such that $h(x, \sigma)>0$ if $\sigma>0$, and
$\dfrac{h(x,\sigma)}{\sigma}$ is decreasing. Let $u,v\in
W^{s,p}_{\beta, 0}(\O)$ be such that $u,v>0$ in $\O$ and
$$\left\{
\begin{array}{rcl}
(-\Delta)^s_{p, \beta}u &\geq & h(x, u)\mbox{  in  }\O,\\ \\
(-\Delta)^s_{p, \beta}v & \le & h(x,v)\mbox{  in   }\O.
\end{array}
\right.
$$
Then, $u\geq v$ in $\Omega$.
\end{Lemma}

The following algebraic inequalities can be proved using suitable rescaling argument.
\begin{Lemma}\label{algg}
Assume that $p\ge 1$, $a, b \in \re^+$ and $\a>0$. Then there exit a positive constants $c, c_1, c_2$, such that
\begin{equation}\label{alge1}
(a+b)^\a\le c_1a^\a+c_2b^\a
\end{equation}
and
\begin{equation}\label{alge3}
|a-b|^{p-2}(a-b)(a^{\a}-b^{\a})\ge c|a^{\frac{p+\a-1}{p}}-b^{\frac{p+\a-1}{p}}|^p.
\end{equation}
In the case where $\a\ge 1$, then under the same conditions on $a,b,p$ as above, we have
\begin{equation}\label{alge2}
|a+b|^{\a-1}|a-b|^{p}\le c |a^{\frac{p+\a-1}{p}}-b^{\frac{p+\a-1}{p}}|^p.
\end{equation}
\end{Lemma}
Since we are considering solution with datum in $L^1$, we need to use the concept of truncation. Recall that, for $k>0$,
\begin{gather*}\label{f-trun}
    T_k(a)=\left\{\begin{array}{cl}
    a\,,&\hbox{ if }|a|\le k\,;\\[2mm]
    k\frac{a}{|a|}\,,&\hbox{ if }|a|> k.
    \end{array}\right.
\end{gather*}
Define $G_k(a)=a-T_k(a)$, taking in consideration the above definition, it is not difficult to show the next algebraic inequalities:
\begin{equation}\label{general}
|a-b|^{p-2}(a-b)(T_k(a)-T_k(b))\ge |T_k(a)-T_k(b)|^p
\end{equation}
and
\begin{equation}\label{general00}
|a-b|^{p-2}(a-b)(G_k(a)-G_k(b))\ge |G_k(a)-G_k(b)|^p,
\end{equation}
where $a,b\in \re$ and $p\ge 1$.

In the same way we will use the classical weighted Marcinkiewicz spaces.
\begin{Definition}
\noindent For a measurable function $u$ we set $$ \Phi _{u}(k)=\mu
\{x\in \O:|u(x)|>k\},$$ where $d\mu=|x|^{-2\beta}dx$.\newline We
say that $u$ is in the Marcinkiewicz space ${\mathcal{M}
}^{q}(\O,d\mu)$ if $ \Phi _{u}(k)\le Ck^{-q}$.
%\newline
Since $\O$ is a bounded domain, then
$$L^{q}(\O,d\mu)\subset {\mathcal{M}
}^{q}(\O,d\mu)\subset L^{q-\e}(\O,d\mu)$$ for all $\e>0$.
\end{Definition}
Since we are considering problem with general datum, then we need to precise the concept of solution. We begin by the following definitions.
\begin{Definition}\label{def:truncature}
Let $u$ be a measurable function, we say that $u\in {\mathcal{T} }^{1,p}_{\b, 0}(\O)$ if for all $k>0$, $T_k(u)\in W^{s,p}_{\beta, 0}(\O)$.
\end{Definition}
Now, we are able to state the sense in which we will take a solution to problem \eqref{eq:def}.

\begin{Definition}\label{sense}
Assume that $f\in L^{1}(\O)$. We say that $u$ is a weak solution to problem \eqref{eq:def} if for all $\phi\in \mathcal{C}^\infty_0(\Omega)$, we have
$$
\frac 12\dyle \iint_{D_\O}
\,|u(x)-u(y)|^{p-2}(u(x)-u(y))(\phi(x)-\phi(y))d\nu=\io f(x)\phi(x)dx.
$$
\end{Definition}

Following \cite{ABB}, we define the notion of entropy solution by
\begin{Definition}\label{def:entropy}
Consider $f\in L^{1}(\O)$, we say that $u\in {\mathcal{T} }^{1,p}_{0,\beta}(\O)$
is an entropy solution to problem \eqref{eq:def} if
\begin{equation}\label{entro001}
\iint_{R_h}|u(x)-u(y)|^{p-1}d\nu\to 0\mbox{   as   }h\to \infty,
\end{equation}
where
$$
R_h=\bigg\{(x,y)\in \ren\times \ren: h+1\le \max\{|u(x)|,|u(y)| \}\mbox{  with  } \min\{|u(x)|,|u(y)| \}\le h\mbox{  or  }u(x)u(y)<0\bigg\}
,$$
and for all $k>0$ and $\varphi\in W^{s,p}_{\beta, 0}(\O) \cap L^{\infty}(\O)$,  we
have
\begin{equation}\label{eq:alcala}
\begin{array}{lll}
&\dyle \frac 12\iint_{D_\O}
\,|u(x)-u(y)|^{p-2}(u(x)-u(y))[T_k(u(x)-\varphi(x))-T_k(u(y)-\varphi(y))]d\nu\le \\
&\dyle \io f(x)T_k(u(x)-\varphi(x)) \, dx.
\end{array}
\end{equation}

\end{Definition}

\begin{remarks}

Notice that for $h>>k$, choosing $\varphi=T_{h-1}(u)$, we obtain that
\begin{equation*}
\begin{array}{lll}
&\dyle \frac 12\iint_{D_\O}
\,|u(x)-u(y)|^{p-2}(u(x)-u(y))[T_k(G_{h-1}(u(x)))-T_k(G_{h-1}(u(y)))]d\nu\le \\
&\dyle \io f(x)T_k(G_{h-1}(u(x))) \, dx\le k\int_{|u|>h-k-1} |f(x)|dx.
\end{array}
\end{equation*}
Since $|u(x)-u(y)|^{p-2}(u(x)-u(y))[T_k(G_{h-1}(u(x)))-T_k(G_{h-1}(u(y)))]\ge 0$ in $D_\O$, then setting
$$
\widetilde{R}_h=\bigg\{(x,y)\in \ren\times \ren: u(x)u(y)\ge 0 \mbox{  with  }|u(x)|\ge h\mbox{  and } h-k-1\le |u(y)|\le h\bigg\}
,$$
and
$$
\widehat{R}_h=\bigg\{(x,y)\in \ren\times \ren: u(x)u(y)\ge 0 \mbox{  with  }|u(x)|\ge h \mbox{  and } h-k-1\le |u(x)|\le h\bigg\}
,$$
we reach that
\begin{equation}\label{two1}
\dyle \frac 12\iint_{\widetilde{R}_h}
\,|u(x)-u(y)|^{p-1}(h-u(y))d\nu \le k\int_{|u|>h-k-1} |f(x)|dx,
\end{equation}
and
\begin{equation}\label{two2}
\dyle \frac 12\iint_{\widehat{R}_h}
\,|u(x)-u(y)|^{p-1}(h-u(x))d\nu \le k\int_{|u|>h-k-1} |f(x)|dx.
\end{equation}
It is clear that
\begin{equation}\label{two11}
\dyle \frac 12\iint_{\{h-k-1\le u(y)<u(x)\le h\}}
\,(u(x)-u(y))^{p}d\nu \le k\int_{|u|>h-k-1} |f(x)|dx,
\end{equation}
and
\begin{equation}\label{two21}
\dyle \frac 12\iint_{\{h-k-1\le u(x)<u(y)\le h\}}
\,(u(y)-u(x))^{p}d\nu \le k\int_{|u|>h-k-1} |f(x)|dx.
\end{equation}
\end{remarks}

%In the same way we define the space $Y^{s,\beta}_0(\Omega)$ as
%the completion of $\mathcal{C}^\infty_0(\Omega)$ with respect to
%previous norm.
%If, moreover,  $\Omega$ is a bounded domain, then we can consider
%$Y^{s,\beta}_0(\Omega)$ with the equivalent norm
%$$
%|||\phi|||_{Y^{s,\beta}_0(\Omega)}=
%\Big(\dint_{\Omega}\dint_{\Omega}|\phi(x)-\phi(y)|^pd\nu\Big)^{\frac
%1p}.
%$$
%It is clear that for all $w,v\in X^{s,p,\beta}(\ren)$, we have
%$$
%\langle L_{\beta}w,v\rangle_{X^{s,p,\beta}(\ren)}
%= a_{N,s}\mbox{ P.V. } \dint_{\ren}\dint_{\ren}|w(x)-w(y)|^{p-2}(w(x)-w(y))(v(x)-v(y))\,d\nu.
%$$

%We have also, for  $p\ge 2$  for all $a\ge b$ two real positive constants
%, then there exist $c>0$ such that
%\begin{equation}\label{alge4}
%a^{p-1}-(a-b)^{p-1}\ge c b^{p-1}
%\end{equation}

\section{Existence Results: Proofs of Theorems \ref{mainth} and \ref{entropi}.}\label{sec2}

In this section we consider the next problem
\begin{equation}\label{main00}
\left\{
\begin{array}{rcll}
(-\Delta)^s_{p, \beta} u & = &f & \mbox{ in }\O,\\
 u &= & 0 & \mbox{ in }  \ren\setminus\O,
\end{array}
\right.
\end{equation}
where $f\in L^1(\O)$.

The main goal of this section is to show that problem \eqref{main00} has a weak solution $u$ in the sense of Definition \ref{sense}. As in the local case, the main idea is to proceed by approximation and then pass to the limit using suitable apriori estimates.

Before proving the main existence results, we need several lemmas.

Let $\{f_n\}_n\subset L^\infty(\Omega)$ be such that $f_n\to f$ strongly in $L^1(\Omega)$ and define $u_n$ as the unique
solution to the approximated problem
\begin{equation}\label{pro:lineal1}
\left\{\begin{array}{rcll}
(-\Delta)^s_{p, \beta}u_n &= & f_n(x) & \mbox{  in  }\O,\\
u_n &= & 0 & \mbox{ in } \ren\backslash\O.
\end{array}
\right.
\end{equation}
Notice that the existence and the uniqueness of $u_n$ follows
using classical variational argument in the space $W^{s,p}_{\beta, 0}(\Omega)$.

The first a priori estimate is given by the next Lemma.
\begin{Lemma}\label{one}
Let $\{u_n\}_n$ be defined as above, then $\{u_n\}_n$ is bounded in the
space $\mathcal{M}^{p_1}(\O,d\mu)$ with
$p_1=\frac{(p-1)N}{N-ps}$.
\end{Lemma}

\pf Using $T_k(u_n)$ as a test function in \eqref{pro:lineal1}, we
reach that
$$\frac 12\iint_{D_{\O}} \,|u_n(x)-u_n(y)|^{p-2}(u_n(y)-u_n(x))[T_k(u_n(x))-T_k(u_n(y))]d\nu\le k\io |f_n(x)| \, dx.$$
Thus
$$
\dyle \iint_{D_{\O}} \,
|u_n(x)-u_n(y)|^{p-2}(u_n(y)-u_n(x))[T_k(u_n(x))-T_k(u_n(y))]d\nu \le C
k.
$$
Recall that $u_n=T_k(u_n)+G_k(u_n)$, then by inequality
\eqref{general}, we have
\begin{equation}\label{eq:eq1}
\frac{1}{k}\iint_{D_{\O}} \,|T_k(u_n(x))-T_k(u_n(y))|^{p}d\nu \le M, \mbox{  for  all  }k>0.
\end{equation}
Now, using the weighted Sobolev inequality in Theorem \ref{Sobolev}, we
get
$$ S\left(\dyle \int_{\ren}
|T_k(u)|^{p^*_{s}}|x|^{-2\beta\frac{p^*_{s}}{p}}
dx\right)^{p/p^*_{s}}\le \iint_{D_{\O}}
\,|T_k(u_n(x))-T_k(u_n(y))|^{p}d\nu\le C k.$$ Since
$\{|u_n|\ge k\}=\{|T_k(u_n)|= k\}$, we obtain that $$ \mu\{x\in \O
:|u_n|\ge k\}\le \mu \{x\in \O:|T_k(u_n)|=k\}\le \io
\dfrac{|T_k(u_n)|^{p^*_{s}}}{k^{p^*_{s}}}|x|^{-2\beta\frac{p^*_{s}}{p}}
dx .$$ Therefore, $ \mu\{x\in \O :|u_n|>k\}\le C
M^{\frac{p^*_{s}}{p}}k^{-(p^*_{s}-\frac{p^*_{s}}{p})}$. Setting $p_1=p^*_{s}-\frac{p^*_{s}}{p}=\frac{N(p-1)}{N-ps}$, we conclude that the sequence $\{u_n\}_n$ is bounded in the space ${\mathcal{M} }^{p_1}(\O,d\mu)$ and the result follows.
\cqd

As a consequence we easily get that the sequence $\{|u_n|^{p-2}u_n\}_n$ is bounded in the space $L^{\sigma}(\O,d\mu)$ for all $\sigma<\frac{N}{N-ps}$.

As in the local case, we prove now that the sequence $\{u_n\}_n$
is bounded in a suitable fractional Sobolev space. More precisely
we have
\begin{Lemma}\label{two}
Assume that $\{u_n\}_n$ is as above, then
\begin{equation}\label{two00}
\int_\O\int_{\O}\dfrac{|u_n(x)-u_n(y)|^{q}}{|x-y|^{N+qs_1}}\dfrac{1}{|x|^\b|y|^\b}dy
\ dx\le M, \mbox{   for all
}q<\dfrac{N(p-1)}{N-s}\mbox{  and for all
}s_1<s.
\end{equation}
\end{Lemma}
\begin{pf}
Let $q<\dfrac{N(p-1)}{N-s}$ be fixed. Since $\O$ is a bounded domain, then it is sufficient to prove \eqref{two00} for $s_1$ very close to $s$. In particular, we fix $s_1$ such that
\begin{equation}\label{twos1}
\frac{pq(s-s_1)}{p-q}<\beta.
\end{equation}
Define $w_n(x)=1-\dfrac{1}{(u^+_n(x)+1)^{\a}}$ where $\a>0$ to be chosen
later and $u^+_n(x)=\max\{u_n(x), 0\}$. Using $w_n$ as a test function in \eqref{pro:lineal1}, we get
$$\frac 12\iint_{D_{\O}} \,|u_n(x)-u_n(y)|^{p-2}(u_n(x)-u_n(y))\frac{(u^+_n(x)+1)^{\a}-(u^+_n(y)+1)^{\a}}{(u^+_n(x)+1)^{\a}(u^+_n(y)+1)^{\a}}d\nu\le\io f_n(x) \, dx.$$
Hence
$$\iint_{D_{\O}} \,|u_n(x)-u_n(y)|^{p-2}(u_n(x)-u_n(y))\frac{(u^+_n(x)+1)^{\a}-(u^+_n(y)+1)^{\a}}{(u^+_n(x)+1)^{\a}(u^+_n(y)+1)^{\a}}d\nu\le
C.$$ Let $v_n(x)=u^+_n(x)+1 $, since
\begin{eqnarray*}
&|u_n(x)-u_n(y)|^{p-2}(u_n(x)-u_n(y))\Big((u^+_n(x)+1)^{\a}-(u^+_n(y)+1)^{\a}\Big)\ge
\\
&
|u^+_n(x)-u^+_n(y)|^{p-2}(u^+_n(x)-u^+_n(y))\Big((u^+_n(x)+1)^{\a}-(u^+_n(y)+1)^{\a}\Big)=\\
&|v_n(x)-v_n(y)|^{p-2}(v_n(x)-v_n(y))\Big(v^\a_n(x)-v^\a_n(y)\Big),
\end{eqnarray*}
it follows that
$$\iint_{D_{\O}} \,|v_n(x)-v_n(y)|^{p-2}(v_n(x)-v_n(y))(\frac{v^\a_n(x)-v^\a_n(y)}{v^\a_n(x)v^\a_n(y)})d\nu\le C. $$
Now, using the fact that $v_n\ge 1$ and by inequality \eqref{alge3}, we get
\begin{equation}\label{hola}
\iint_{D_{\O}}
\,\frac{|v_n^{\frac{p+\a-1}{p}}(x)-v_n^{\frac{p+\a-1}{p}}(y)|^{p}}{v^\a_n(x)v^\a_n(y)}d\nu\le C .
\end{equation}
Define $q_1=q\frac{s_1}{s}<q$, using H\"older inequality, it follows that
\begin{eqnarray*}
&\dyle \int_\O\int_{\O}\dfrac{|v_n(x)-v_n(y)|^{q}}{|x-y|^{N+qs_1}}\dfrac{dy\ dx}{|x|^\b|y|^\b}=\\
&\dyle \int_\O\int_{\O}\dfrac{|v_n(x)-v_n(y)|^{q}}{|x-y|^{qs}}\frac{(v_n(x)+v_n(y))^{\a-1}}{(v_n(x)v_n(y))^{\a}}\frac{(v_n(x)v_n(y))^{\a}}{(v_n(x)+v_n(y))^{\a-1}}|x-y|^{(q-q_1)s}\dfrac{dy\ dx}{|x|^\b|y|^\b|x-y|^{N}}\\ \\
&\dyle\le\Big(\int_\O\int_{\O}\dfrac{|v_n(x)-v_n(y)|^{p}(v_n(x)+v_n(y))^{\a-1}}{|x-y|^{N+ps}(v(x)v(y))^{\a}|x|^\b|y|^\b}dy
\ dx \Big)^{\frac{q}{p}}\\
&\times \dyle \Big(\int_\O\int_{\O}\frac{(v_n(x)+v_n(y))^{\a-1}}{(v(x)v(y))^{\a}}\frac{(v_n(x)v_n(y))^{\a\frac{p}{p-q}}}{(v_n(x)+v_n(y))^{(\a-1)\frac{p}{p-q}}}|x-y|^{(q-q_1)s\frac{p}{p-q}}\dfrac{dy \ dx}{|x-y|^{N}|x|^\b|y|^\b}\Big)^{\frac{p-q}{q}}.\end{eqnarray*}
Now, using the algebraic inequality \eqref{alge2}, it follows that
$$
|v_n(x)-v_n(y)|^{p}(v_n(x)+v_n(y))^{\a-1}\le C |v_n(x)^{\frac{p+\a-1}{p}}-v_n(y)^{\frac{p+\a-1}{p}}|^{p}.
$$
Hence, taking in consideration that $\O\times \O\subset D_{\O}$ and by \eqref{hola}, we get
\begin{eqnarray*}
&\dyle \Big(\int_\O\int_{\O}\dfrac{|v_n(x)-v_n(y)|^{p}(v_n(x)+v_n(y))^{\a-1}}{|x-y|^{N+ps}(v(x)v(y))^{\a}|x|^\b|y|^\b}
dy \ dx\Big)^{\frac{q}{p}}\\
& \le \dyle C\Big(\iint_{D_{\O}}\frac{|v_n(x)^{\frac{p+\a-1}{p}}-v_n(y)^{\frac{p+\a-1}{p}}|^{p}}{|x-y|^{N+ps}(v_n(x)v_n(y))^{\a}|x|^\b|y|^\b}
dy \ dx\Big)^{\frac{q}{p}}\le C.
\end{eqnarray*}
So we get
\begin{eqnarray*}
&\dyle \iint_{D_{\O}}\dfrac{|v_n(x)-v_n(y)|^{q}}{|x-y|^{N+qs_1}}\dfrac{dy \ dx}{|x|^\b|y|^\b}\le\\
&\dyle c\Big(\int_\O\int_{\O}\Big(\frac{(v_n(x)v_n(y))^{\a}}{(v_n(x)+v_n(y))^{\a}}\Big)^\frac{q}{p-q}(v_n(x)+v_n(y))^\frac{q}{p-q}\dfrac{1}{|x-y|^{N-\frac{ps(q-q_1)}{p-q}}}\dfrac{dy
\ dx}{|x|^\b|y|^\b}\Big)^{\frac{p-q}{q}}.
\end{eqnarray*}
By inequality \eqref{alge1}, we have
$$(v_n(x)+v_n(y))\big( \frac{v_n(x)v_n(y)}{v_n(x)+v_n(y)}\big)^{\a}\le c (v_n(x)+v_n(y))^{\a+1}\le c_1v^{\a+1}_n(x)+c_2v^{\a+1}_n(y).$$
Therefore,
\begin{eqnarray*}
&\dyle \int_\O\int_{\O}\dfrac{|v_n(x)-v_n(y)|^{q}}{|x-y|^{N+qs_1}}\dfrac{dx \ dy}{|x|^\b|y|^\b}\le \\ &\dyle c_1\Big(\int_\O\int_{\O}\dfrac{v_n^{\frac{(\a+1)q}{p-q}}(x)dx \ dy}{|x-y|^{N-\frac{ps(q-q_1)}{p-q}}|x|^\b|y|^\b}\Big)^{\frac{p-q}{q}} +  c_2\Big(\int_\O\int_{\O}\dfrac{v_n^{\frac{(\a+1)q}{p-q}}(y)dx \ dy}{|x-y|^{N-\frac{ps(q-q_1)}{p-q}}|x|^\b|y|^\b}\Big)^{\frac{p-q}{q}}.
\end{eqnarray*}
We treat each term separately.
$$\int_\O\int_{\O}\dfrac{v^{\frac{(\a+1)q}{p-q}}_n(x)dx \ dy}{|x-y|^{N-\frac{ps(q-q_1)}{p-q}}|x|^\b|y|^\b}=\int_{\O} \dfrac{v^{\frac{(\a+1)q}{p-q}}_n(x)}{|x|^\b}dx\int_{\O}\dfrac{dy}{|x-y|^{N-\frac{ps(q-q_1)}{p-q}}|y|^\b}.$$
Since $\Omega$ is a bounded domain, then $\Omega\subset\subset B_R(0)$. Thus
$$\int_\O\int_{\O}\dfrac{v^{\frac{(\a+1)q}{p-q}}_n(x)dx \ dy}{|x-y|^{N-\frac{ps(q-q_1)}{p-q}}|x|^\b|y|^\b}\le \int_{B_R(0)} \dfrac{v^{\frac{(\a+1)q}{p-q}}_n(x)}{|x|^\b}dx\int_{B_R(0)}\dfrac{dy}{|x-y|^{N-\frac{ps(q-q_1)}{p-q}}|y|^\b}$$
where $v_n=1$ in $B_R(0)\setminus \Omega$.
We set $r=|x|$ and $\rho=|y|$, then $x=rx', y=\rho y'$.
where $|x'|=|y'|=1$,

Let $\tau=\frac{(\a+1)q}{p-q}$ and $\theta=\frac{ps(q-q_1)}{p-q}$, then
$$\int_\O\int_{\O}\dfrac{v^\tau_n(x)dx \ dy}{|x-y|^{N-\theta}|x|^\b|y|^\b}\le
\int_{B_R(0)}\dfrac{v^\tau_n(x)\ dx }{|x|^{\beta}}
\dint\limits_0^{R}\dfrac{\rho^{N-1}}{\rho^{\beta}
r^{N-\theta}}\left(
\dint\limits_{|y'|=1}\dfrac{dH^{n-1}(y')}{|x'-\frac{\rho}{r}
y'|^{N-\theta}} \right) \,d\rho.$$ We set
$\sigma=\dfrac{\rho}{r}$, hence
\begin{eqnarray*}
\dyle \int_\O\int_{\O}\dfrac{v^\tau_n(x)dx \ dy}{|x-y|^{N-\theta}|x|^\b|y|^\b} &\le &
\int_{B_R(0)}\dfrac{v^\tau_n(x)\ dx
}{|x|^{2\beta-\theta}}\dint\limits_0^{\frac{R}{r}}\sigma^{N-\beta-1}
\left(\dint\limits_{|y'|=1}\dfrac{dH^{n-1}(y')}{|x'-\s
y'|^{N-\theta}} \right) \,d\sigma\\
&\le & \dyle \int_{B_R(0)}\dfrac{v^\tau_n(x)\ dx
}{|x|^{2\beta-\theta}}\dint\limits_0^{\infty}\sigma^{N-\beta-1}
\left(\dint\limits_{|y'|=1}\dfrac{dH^{n-1}(y')}{|x'-\s
y'|^{N-\theta}} \right) \,d\sigma.
\end{eqnarray*}
Define
$$
K(\s)=\dint\limits_{|y'|=1}\dfrac{dH^{n-1}(y')}{|x'-\s
y'|^{N-\theta}},
$$
as in \cite{FV}, we get
\begin{equation}\label{kkk}
K(\sigma)=2\frac{\pi^{\frac{N-1}{2}}}{\beta(\frac{N-1}{2})}\int_0^\pi
\frac{\sin^{N-2}(\xi)}{(1-2\sigma \cos
(\xi)+\sigma^2)^{\frac{N-\theta}{2}}}d\xi.
\end{equation}
Since $K(\s)\le C|1-\s|^{-1+\theta}$
as $\s\to 1$ and using the fact that $\s^{N-1-\beta}K(\s)\simeq \s^{-1-\beta+\theta}$ as $\s\to \infty$ with $\theta<\beta$(that follows by \eqref{twos1}),
we obtain that $\int_0^\infty \s^{N-1-\beta}K(\s)d\s\equiv
C_3<\infty$. Therefore
$$\int_\O\int_{\O}\dfrac{v^\tau_n(x)dy \ dx}{|x-y|^{N-\theta}|x|^\b|y|^\b}\le C_3\int_{B_R(0)}\dfrac{v^\tau_n(x)\ dx }{|x|^{2\beta-\theta}}.$$
Since $q<\frac{(p-1)N}{N-s}$, we can choose $\a>0$ such that $\tau<\frac{(p-1)N}{N-ps}$. Using Lemma \ref{one}, we reach that
$$
\int_{B_R(0)}\dfrac{v^\tau_n(x)\ dx }{|x|^{2\beta-\theta}}\le C\mbox{  for all  n}.
$$
Hence we conclude that
$$\int_\O\int_{\O}\dfrac{|u^+_n(x)-u^+_n(y)|^{q}}{|x-y|^{N+qs_1}}\dfrac{dy \ dx}{|x|^\b|y|^\b}\le C.$$
In the same way and using $\Big(1-\dfrac{1}{(u^-_n(x)+1)^{\a}}\Big)$ as a
test function in \eqref{pro:lineal1}, we obtain that
$$\int_\O\int_{\O}\dfrac{|u^-_n(x)-u^-_n(y)|^{q}}{|x-y|^{N+qs_1}}\dfrac{dy \ dx}{|x|^\b|y|^\b}\le C.$$
Combining the above estimates, we reach that
$$\int_\O\int_{\O}\dfrac{|u_n(x)-u_n(y)|^{q}}{|x-y|^{N+qs_1}}\dfrac{dy \ dx}{|x|^\b|y|^\b}\le C$$
and the result follows.
\end{pf}

\begin{remark}\label{r01}
As a consequence we get the existence of a measurable function $u$
such that $T_k(u)\in W^{s,p}_{\beta, 0}(\Omega)$, $|u|^{p-2}u\in
L^\s(\O,|x|^{-2\beta} dx)$ for all $\s<\dfrac{N}{N-ps}$ and
$T_k(u_n)\rightharpoonup T_k(u)$ weakly in $W^{s,p}_{\beta, 0}(\Omega)$.

It is clear that $u_n\to u$ a.e. in $\O$. Since $u_n=0$ a.e. in $\ren\backslash \O$, then $u=0$ a.e. in $\ren\backslash \O$.

Notice that by Lemma \ref{one} we conclude that
$$
|u_n|^{p-2}u_n\to |u|^{p-2}u \mbox{   strongly in   }L^a(\Omega, d\mu)\mbox{  for all   }a<\dfrac{N}{N-ps}.
$$
Let
\begin{equation}\label{nota}
U_n(x,y)=|u_n(x)-u_n(y)|^{p-2}(u_n(x)-u_n(y))\mbox{  and  }U(x,y)=|u(x)-u(y)|^{p-2}(u(x)-u(y)).
\end{equation}
Since $\O$ is a bounded domain, then by the result of Lemma \ref{two} and using Vitali's Lemma,  we obtain that
$$
U_n\to U \mbox{   strongly in   }L^1(\O\times \O, d\nu).
$$
\end{remark}

We are now able to prove the first existence result.

\

{\bf Proof of Theorem \ref{mainth}}.

It is clear that estimate \eqref{estimm} follows using Lemma \ref{two} and Fatou's Lemma.

Let $\phi\in \mathcal{C}^\infty_0(\Omega)$, then using $\phi$ as a test function in \eqref{pro:lineal1}, it follows that

\begin{equation}\label{TR}\frac 12\iint_{D_{\O}} \,|u_n(x)-u_n(y)|^{p-2}(u_n(x)-u_n(y))(\phi(x)-\phi(y))d\nu =\io f_n(x)\phi(x) \, dx.\end{equation}
We set $\Phi(x,y)=\phi(x)-\phi(y)$. By \eqref{TR}, we have
\begin{equation}\label{GGG}
\frac 12\iint_{D_{\O}} \,U(x,y)\Phi(x,y)d\nu+\frac 12\iint_{D_{\O}}\Big(U_n(x,y)-U(x,y)\Big)\Phi(x,y)d\nu=\io f_n(x)\phi(x) \, dx.
\end{equation}
It is clear that
$$
\io f_n(x)\phi(x) \, dx\to \io f(x)\phi(x) \, dx \mbox{   as   }n\to \infty.
$$
We claim that
$$
\iint_{D_{\O}}\Big(U_n(x,y)-U(x,y)\Big)\Phi(x,y)d\nu\to 0\mbox{   as    }n\to \infty.
$$
Since $u_n\to u$ a.e. in $\Omega$, it follows that
$$
\dfrac{U_n(x,y)\Phi(x,y)}{|x-y|^{N+ps}|x|^\beta|y|^\beta}\to \dfrac{U(x,y)\Phi(x,y)}{|x-y|^{N+ps}|x|^\beta|y|^\beta}\:a.e. \mbox{  in  }D_\Omega.
$$
Using the fact that $u(x)=u_n(x)=\phi(x)=0$ for all $x\in \ren\backslash\O$, we reach that
$$
\int_{\ren\backslash \O}\int_{\ren\backslash \O} (U_n(x,y)-U(x,y))\Phi(x,y)d\nu=0.
$$
Thus
\begin{eqnarray*}
& \dyle\iint_{D_{\O}}(U_n(x,y)-U(x,y))\Phi(x,y)d\nu=\iint_{\O\times \O}(U_n(x,y)-U(x,y))\Phi(x,y)d\nu\\
&+\dyle \int_{\ren\backslash \O}\int_{\O}(U_n(x,y)-U(x,y))\Phi(x,y)d\nu+\int_{\O}\int_{\ren\backslash \O} (U_n(x,y)-U(x,y))\Phi(x,y)d\nu\\
&=I_1(n)+I_2(n)+I_3(n).
\end{eqnarray*}
Using Lemma \ref{two} and Remark \ref{r01}, we easily get that
$$
I_1(n)\to 0\mbox{   as   }n\to \infty.
$$
We deal now with $I_2(n)$. It is clear that
$$
|(U_n(x,y)-U(x,y))\Phi(x,y)|\le (|u_n(x)|^{p-1}+|u(x)|^{p-1})|\phi(x)| \mbox{  in   }\O\times B_R\backslash \O.
$$
Since
$$
\sup_{\{x\in \text{Supp}\phi, \:\:y\in B_R\backslash \O\}}\dfrac{1}{|x-y|^{N+ps}}\le C,
$$
then
$$
\Big|\dfrac{(U_n(x,y)-U(x,y))\Phi(x,y)}{|x-y|^{N+ps}|x|^\beta|y|^\beta}\Big|\le \dfrac{(|u_n(x)|^{p-1}+|u(x)|^{p-1})|\phi(x)|}{|x|^\beta|y|^\beta}\equiv Q_n(x,y).
$$
Using Lemma \ref{one} and Remark \ref{r01}, we get $Q_n\to Q$ strongly in $L^1(\O\times B_R\backslash \O)$ with $$Q(x,y)=2|u(x)|^{p-1}|\phi(x)|{|x|^{-\beta}|y|^{-\beta}}.$$ Thus using the Dominated Convergence Theorem we reach that $I_2(n)\to 0\mbox{   as   }n\to \infty$.
In the same way we obtain that $I_3(n)\to 0\mbox{   as   }n\to \infty$. Hence the claim follows.

Therefore, passing to the limit in \eqref{GGG}, there results that
$$
\frac 12\iint_{D_{\O}} \,U(x,y)\Phi(x,y)d\nu=\io f(x)\phi(x) \, dx.
$$
Hence we conclude. \cqd

\begin{remark}

\

\begin{enumerate}

\item It is clear that the same existence result holds if we replace $f$ by a bounded Radon measure $\nu$.

\item In the case where $\beta=0$, we get the same existence and regularity results obtained in \cite{KMS}.
\end{enumerate}
\end{remark}

\subsection{The case of positive datum: Existence and uniqueness of the positive entropy solution}\label{subb}

\

If $f\gneq 0$, we choose $f_n=T_n(f)$, thus $\{u_n\}_n$ is an increasing sequence. In this case we are able to prove that problem \eqref{main00} has a unique  entropy positive solution in the sense of Definition \ref{def:entropy}. Before staring the proof of Theorem \ref{entropi}, let us prove the next result.
\begin{Lemma}\label{convv}
Let $\{u_n\}_n$ and $u$ defined above, then
$$T_k(u_n)\longrightarrow T_k(u) \   \   \qquad  \text{ strongly in } W^{s,p}_{\beta, 0}(\Omega).$$
\end{Lemma}
The proof of Lemma \ref{convv} will be a consequence of the next more general compactness result.
\begin{Lemma}\label{compact}
Let $\{u_n\}_n\subset W^{s,p}_{\beta, 0}(\Omega)$ be an increasing sequence such that $u_n\ge 0$ and
$(-\Delta)^s_{p, \beta} u_n\ge 0$. Assume that $\{T_k(u_n)\}_n$ is bounded in $ W^{s,p}_{\beta, 0}(\Omega)$ for all $k>0$, then there exists a measurable function $u$ such that $u_n\uparrow u$ a.e. in $\O,\:\: T_k(u)\in W^{s,p}_{\beta, 0}(\Omega)$ for all $k>0$ and
\begin{equation}\label{ty}
T_k(u_n)\longrightarrow T_k(u)\:\:  \text{ strongly in } W^{s,p}_{\beta, 0}(\Omega).
\end{equation}
\end{Lemma}
\begin{pf}
Since $\{T_k(u_n)\}_n$ is bounded in $ W^{s,p}_{\beta, 0}(\Omega)$, then using the monotony of the sequence $\{u_n\}_n$, we get the existence of a measurable function $u$ such that $u_n\uparrow u$ a.e. in $\O, T_k(u)\in W^{s,p}_{\beta, 0}(\Omega)$ and $T_k(u_n)\rightharpoonup T_k(u)$ weakly in $W^{s,p}_{\beta, 0}(\Omega)$.
Since $(-\Delta)^s_{p, \beta} u_n\ge 0$, then
$$
\langle (-\Delta)^s_{p, \beta} u_n, T_k(u_n)-T_k(u)\rangle\le 0.
$$
Thus
\begin{eqnarray*}
& \dyle \iint_{D_{\O}} \,|u_n(x)-u_n(y)|^{p-2}(u_n(x)-u_n(y))(T_k(u_n(x))-T_k(u_n(y)))d\nu\\
&\le \dyle\iint_{D_{\O}} \,|u_n(x)-u_n(y)|^{p-2}(u_n(x)-u_n(y))(T_k(u(x))-T_k(u(y)))d\nu.
\end{eqnarray*}
Define
$$
I_{1,n}\equiv \dyle \iint_{D_{\O}} \,|u_n(x)-u_n(y)|^{p-2}(u_n(x)-u_n(y))(T_k(u_n(x))-T_k(u_n(y)))d\nu,
$$
and
$$
I_{2,n}\equiv \dyle\iint_{D_{\O}} \,|u_n(x)-u_n(y)|^{p-2}(u_n(x)-u_n(y))(T_k(u(x))-T_k(u(y)))d\nu.
$$
For simplicity of typing we set
$$
\mathrm{T}_{n, k}(x,y)\equiv |T_k(u_n(x))-T_k(u_n(y))|^{p-2}(T_k(u_n(x))-T_k(u_n(y))).
$$
We have
\begin{eqnarray*}
I_{1,n} &= &\iint_{D_{\O}} \,|T_k(u_n(x))-T_k(u_n(y))|^pd\nu\\
& + & \iint_{D_{\O}} \,\Big[U_n(x,y)-\mathrm{T}_{n, k}(x,y)\Big](T_k(u_n(x))-T_k(u_n(y)))d\nu.
\end{eqnarray*}
In the same way, using Young inequality, we obtain that
\begin{eqnarray*}
I_{2,n} &= & \iint_{D_{\O}} \, \mathrm{T}_{n, k}(x,y)(T_k(u(x))-T_k(u(y)))d\nu\\
&+& \iint_{D_{\O}} \,\Big[U_n(x,y)-\mathrm{T}_{n, k}(x,y)\Big](T_k(u(x))-T_k(u(y)))d\nu\\
&\le & \frac{p-1}{p}\iint_{D_{\O}} \,|T_k(u_n(x))-T_k(u_n(y))|^pd\nu+\frac 1p\iint_{D_{\O}} \, |T_k(u(x))-T_k(u(y))|^pd\nu\\
&+& \iint_{D_{\O}} \,\Big[U_n(x,y)-\mathrm{T}_{n, k}(x,y)\Big](T_k(u(x))-T_k(u(y)))d\nu.
\end{eqnarray*}
Combining the above estimates, it follows that
\begin{eqnarray*}
& & \frac{1}{p}\iint_{D_{\O}} \,|T_k(u_n(x))-T_k(u_n(y))|^pd\nu\\
&+ & \iint_{D_{\O}} \,\Big[U_n(x,y)-\mathrm{T}_{n, k}(x,y)\Big]\Big[(T_k(u_n(x))-T_k(u(x)))-(T_k(u_n(y))-T_k(u(y)))\Big]d\nu\\
&\le & \frac 1p\iint_{D_{\O}} \, |T_k(u(x))-T_k(u(y))|^pd\nu.
\end{eqnarray*}
Define
$$
K_n(x,y)\equiv \Big[U_n(x,y)-\mathrm{T}_{n, k}(x,y)\Big]\Big[(T_k(u_n(x))-T_k(u(x)))-(T_k(u_n(y))-T_k(u(y)))\Big].
$$
We claim that $K_n(x,y)\ge 0\:a.e\mbox{  in   }D_{\O}$. To see that we set
$$D_1=\{(x,y)\in D_{\O}: u_n(x)\le k, u_n(y) \le k \}, \  D_2=\{(x,y)\in D_{\O}: u_n(x)\ge k, u_n(y) \ge k \}.$$
$$D_3=\{(x,y)\in D_{\O}: u_n(x)\ge k, u_n(y) \le k \}, \ D_4=\{(x,y)\in D_{\O}: u_n(x)\le k, u_n(y) \ge k \},$$
then
$D_{\O}=D_1\cup D_2\cup D_3\cup D_4$.

In $D_1$, we have $U_n(x,y)-\mathrm{T}_{n, k}(x,y)=0$, then $K_n(x,y)=0$.  In the same way,
if  $(x,y)\in D_2$, we have $u(x)\ge u_n(x)\ge k$ and $u(y)\ge u_n(y)\ge k$, then
$[(T_k(u_n(x))-T_k(u(x)))-(T_k(u_n(y))-T_k(u(y)))\Big]=0$. Thus $K_n(x,y)=0$ in $D_2$.

Assume that $(x,y)\in D_3$, then $$
U_n(x,y)-\mathrm{T}_{n, k}(x,y)=(u_n(x)-u_n(y))^{p-1}-(k-u_n(y))^{p-1}\ge 0.
$$
Since
$$
[(T_k(u_n(x))-T_k(u(x)))-(T_k(u_n(y))-T_k(u(y)))\Big]=-(T_k(u_n(y))-T_k(u(y)))\ge 0,
$$
it follows that $K_n(x,y)\ge 0$ in $D_3$. In the same way we reach that $K_n(x,y)\ge 0$ in $D_4$. Hence the claim follows.
As a conclusion we get
$$
\limsup_{n\to \infty}\iint_{D_{\O}} \,|T_k(u_n(x))-T_k(u_n(y))|^pd\nu\le \iint_{D_{\O}} \,|T_k(u(x))-T_k(u(y))|^pd\nu.
$$
Since $T_k(u_n)\rightharpoonup T_k(u)$ weakly in $W^{s,p}_{\beta, 0}(\Omega)$, then $T_k(u_n)\to T_k(u)$ strongly in $W^{s,p}_{\beta, 0}(\Omega)$ and the result follows.
\end{pf}

\begin{remark}\label{RR}

\

\begin{enumerate}
\item As a consequence of the previous strong convergence we reach that
$$
\iint_{D_{\O}} K_n(x,y)d\nu\to 0\mbox{   as   }n\to \infty.
$$
\item Let $w_n=1-\dfrac{1}{1+u_n}$, using $w_n$ as a test function in \eqref{pro:lineal1},
$$
\iint_{D_\O}
\,\dfrac{|u_n(x)-u_n(y)|^{p}}{(1+u_n(x))(1+u_n(y))}d\nu=\dyle \io f_n(x)w_n(x)dx\to \dyle \io f(x)w(x)dx\mbox{  as }n\to \infty
,$$
where $w=1-\dfrac{1}{1+u}$. For $k>0$ fixed, we define the sets
$$A_n=D_{\O}\cap \{u_n(x)\ge 2k, u_n(y)\le k\}\mbox{  and  } A=D_{\O}\cap \{u(x)\ge 2k, u(y)\le k\}.$$
It is clear that for $(x,y)\in A_n$, we have
$u_n(x)-u_n(y)\ge \frac 12 u_n(x)$. Thus
\begin{equation}\label{RTR}
\iint_{D_{\O}}
u^{p-1}_n(x)\chi_{A_n}(x,y)d\nu\le C(k)
\iint_{D_\O}
\,\dfrac{|u_n(x)-u_n(y)|^{p}}{(1+u_n(x))(1+u_n(y))}d\nu<\bar{C}(k).
\end{equation}
Since $u_n(x)\chi_{\{A_n\}}(x,y)\to u(x)\chi_{\{A\}}\:a.e$ in $D_{\O}$, then if $p>2$, we get
$$
u_n(x)\chi_{A_n}(x,y)  \rightharpoonup u(x)\chi_A \mbox{  weakly  in   }L^{p-1}(D_\O, d\nu).
$$

\item From \eqref{RTR} we conclude that
$$
\nu\{D_{\O}\cap A_n\}\equiv \iint_{D_{\O}\cap A_n}d\nu\le \tilde{C}(k).
$$
Hence by Fatou's lemma, we reach that
$$
\nu\{D_{\O}\cap A\}\equiv \iint_{D_{\O}\cap A}d\nu\le \tilde{C}(k).
$$

\end{enumerate}
\end{remark}

\

Now, we are in position to prove the existence and the uniqueness of the entropy solution.

\

\noindent {\bf{Proof of Theorem \ref{entropi}: Existence part:}}

\

 It is clear that the existence of $u$ follows using Theorem \ref{mainth}, however the strong convergence of $\{T_k(u_n)\}_n$ in the space $W^{s,p}_{\beta, 0}(\Omega)$ is a consequence of Lemma \ref{convv}. To finish we just need to show that $u$ is an entropy solution to problem \eqref{main00} in the sense of Definition \ref{def:entropy}.

Since $u, u_n\ge 0$, then the set $R_h$ given in the Definition \ref{def:entropy} is reduced to
$$
R_h=\bigg\{(x,y)\in \ren\times \ren: h+1\le \max\{u(x),u(y)\}\mbox{  with  } \min\{u(x),u(y)\}\le h\bigg\}. $$
Using $T_1(G_h(u_n))$ as a test function in \eqref{pro:lineal1}, it follows that
\begin{eqnarray*}
&\dfrac 12\dyle \iint_{D_\O}
\,|u_n(x)-u_n(y)|^{p-2}(u_n(x)-u_n(y))[T_1(G_h(u_n(x)))-T_1(G_h(u_n(y)))]d\nu= \\
&\dyle \io f_n(x)T_1(G_h(u_n(x)))\, dx\le \int_{u_n\ge h}f_n(x)dx.
\end{eqnarray*}
It is not difficult to show that, for $(x,y)\in R_h$, we have
$$
|u_n(x)-u_n(y)|^{p-2}(u_n(x)-u_n(y))[T_1(G_h(u_n(x)))-T_1(G_h(u_n(y)))]\ge 0.
$$
Thus, using Fatou's lemma, we conclude that
\begin{eqnarray*}
&\dfrac 12\dyle \iint_{D_\O}
\,|u(x)-u(y)|^{p-2}(u(x)-u(y))[T_1(G_h(u(x)))-T_1(G_h(u(y)))]d\nu\le \\
&\dyle \liminf_{n\to \infty} \dfrac 12\dyle \iint_{D_\O}
\,|u_n(x)-u_n(y)|^{p-2}(u_n(x)-u_n(y))[T_1(G_h(u_n(x)))-T_1(G_h(u_n(y)))]d\nu\\
&\le \dyle \io f(x)T_1(G_h(u(x)))\, dx\le \int_{u\ge h}f(x)dx.
\end{eqnarray*}
It is clear that for all $(x,y)\in R_h$, we have
$$
|u(x)-u(y)|^{p-2}(u(x)-u(y))[T_1(G_h(u(x)))-T_1(G_h(u(y)))]\ge |u(x)-u(y)|^{p-1},
$$
therefore, using the fact that
$$
\int_{u\ge h}f(x)dx\to 0\mbox{   as   }h\to \infty,
$$
we conclude that
$$
\iint_{R_h}
\,|u(x)-u(y)|^{p-1}d\nu \to 0\mbox{   as   }h\to \infty.
$$
Hence \eqref{entro001} holds.

Recall that
$$
U_n(x,y)=|u_n(x)-u_n(y)|^{p-2}(u_n(x)-u_n(y))\mbox{  and  }U(x,y)=|u(x)-u(y)|^{p-2}(u(x)-u(y)).
$$
Let $v\in W^{s,p}_{\beta, 0}(\O) \cap L^{\infty}(\O)$, taking $T_k(u_n-v)$ as a test function in \eqref{pro:lineal1}, we reach that
\begin{eqnarray*}
&\dfrac 12\dyle \iint_{D_\O}
\, U_n(x,y)[T_k(u_n(x)-v(x))-T_k(u_n(y)-v(y))]d\nu= \\
&\dyle \io f_n(x)T_k(u_n(x)-v(x)) \, dx.
\end{eqnarray*}
It is to see that
$$
\dyle \io f_n(x)T_k(u_n(x)-v(x)) \, dx\to \dyle \io f(x)T_k(u(x)-v(x)) \, dx\mbox{   as   }n\to \infty.
$$
We deal now with the first term. We have
$$
U_n(x,y)[T_k(u_n(x)-v(x))-T_k(u_n(y)-v(y))]=:K_{1,n}(x,y)+K_{2,n}(x,y)
,$$
where
\begin{eqnarray*}
K_{1,n}(x,y) &=& |(u_n(x)-v(x))-(u_n(y)-v(y))|^{p-2}((u_n(x)-v(x))-(u_n(y)-v(y)))\\
&\times  & [T_k(u_n(x)-v(x))-T_k(u_n(y)-v(y))],
\end{eqnarray*}
and

\begin{eqnarray*}
&K_{2,n}(x,y)=\\
& \Big[U_n(x,y)-|(u_n(x)-v(x))-(u_n(y)-v(y))|^{p-2}((u_n(x)-v(x))-(u_n(y)-v(y)))\Big]\\ &\times [T_k(u_n(x)-v(x))-T_k(u_n(y)-v(y))].
\end{eqnarray*}
It is clear that $K_{1,n}(x,y)\ge 0$ \ a.e. in $D_\Omega$, since
\begin{eqnarray*}
K_{1,n}(x,y)&\to & |(u(x)-v(x))-(u(y)-v(y))|^{p-2}((u(x)-v(x))-(u(y)-v(y)))\\
&\times  & [T_k(u(x)-v(x))-T_k(u(y)-v(y))] \  a.e. \mbox{  in   }D_{\Omega},
\end{eqnarray*}
as $n\to \infty$, using Fatou's Lemma, we obtain that
\begin{eqnarray*}
&\dyle \iint_{D_{\O}}K_{1,n}(x,y)d\nu\ge\\
&\dyle \iint_{D_{\O}}\Big[(u(x)-v(x))-(u(y)-v(y))|^{p-2}((u(x)-v(x))-(u(y)-v(y)))\Big]\\
&\times[T_k(u(x)-v(x))-T_k(u(y)-v(y))]d\nu.
\end{eqnarray*}
We deal now with $K_{2,n}$.

We set
$$
w_n=u_n-v,\:w=u-v, \:\:\s_1(x,y)=u_n(x)-u_n(y)\mbox{  and  }\s_2(x,y)=w_n(x)-w_n(y).$$ Then
$$
K_{2,n}(x,y)=\Big[|\s_1(x,y)|^{p-2}\s_1(x,y)-|\s_2(x,y)|^{p-2}\s_2(x,y)\Big]\times [T_k(u_n(x)-v(x))-T_k(u_n(y)-v(y))].
$$
We claim that
\begin{equation}\label{cll}
\begin{array}{lll}
&\dyle \iint_{D_{\O}}K_{2,n}(x,y)d\nu\to \\
& \dyle \iint_{D_{\O}}\Big[U(x,y)-|(u(x)-v(x))-(u(y)-v(y))|^{p-2}((u(x)-v(x))-(u(y)-v(y)))\Big]\\ \\ &\times [T_k(u(x)-v(x))-T_k(u(y)-v(y))]d\nu\mbox{    as    }\:\:n\to \infty.
\end{array}
\end{equation}
We divide the proof of the claim into two cases according to the value of $p$.

\

\noindent {\bf \emph{The singular case $p\in (1,2]$:}}
In this case we have
$$
\Big||\s_1(x,y)|^{p-2}\s_1(x,y)-|\s_2(x,y)|^{p-2}\s_2(x,y)\Big|\le C|\s_1(x,y)-\s_2(x,y)|^{p-1}=C|v(x)-v(y)|^{p-1}.
$$
Thus
$$
|K_{2,n}(x,y)|\le C|v(x)-v(y)|^{p-1}|T_k(u_n(x)-v(x))-T_k(u_n(y)-v(y))|\equiv \tilde{K}_{2,n}(x,y).
$$
Using the fact that $T_k(u_n)\to T_k(u)$ strongly in $W^{s,p}_{\beta, 0}(\Omega)$, then since $v\in  W^{s,p}_{\beta, 0}(\O) \cap L^{\infty}(\O)$, we get
$$
\tilde{K}_{2,n}\to C|v(x)-v(y)|^{p-1}|T_k(u(x)-v(x))-T_k(u(y)-v(y))|\mbox{  strongly  in  }L^1(D_\O, d\nu).
$$
Using the Dominated Convergence theorem we reach that
\begin{eqnarray*}
&\dyle \iint_{D_{\O}}K_{2,n}(x,y)d\nu\to \\
& \dyle \iint_{D_{\O}}\Big[U(x,y)-|(u(x)-v(x))-(u(y)-v(y))|^{p-2}((u(x)-v(x))-(u(y)-v(y)))\Big]\\ &\times [T_k(u(x)-v(x))-T_k(u(y)-v(y))]d\nu,
\end{eqnarray*}
as $n\to \infty$ and the claim follows in this case.

\

\noindent {\bf{\emph{The degenerate case $p>2$:}}} This case is more relevant. As in the previous case, we have
\begin{eqnarray*}
& \Big||\s_1(x,y)|^{p-2}\s_1(x,y)-|\s_2(x,y)|^{p-2}\s_2(x,y)\Big|\\
&\le C_1|\s_1(x,y)-\s_2(x,y)|^{p-1}+C_2|\s_2(x,y)|^{p-2}|\s_1(x,y)-\s_2(x,y)|\\
&\le C_1|v(x)-v(y)|^{p-1} +C_2|v(x)-v(y)||w_n(x)-w_n(y)|^{p-2}\\
&\le C_1|v(x)-v(y)|^{p-1} +C_2|v(x)-v(y)||u_n(x)-u_n(y)|^{p-2}.
\end{eqnarray*}
Thus
\begin{eqnarray*}
|K_{2,n}(x,y)|& \le & C_1|v(x)-v(y)|^{p-1}|T_k(u_n(x)-v(x))-T_k(u_n(y)-v(y))|\\
&+ & C_2|v(x)-v(y)||u_n(x)-u_n(y)|^{p-2}|T_k(w_n(x))-T_k(w_n(y))|\\
&\equiv & \bar{K}_{2,n}(x,y)+\check{K}_{2,n}(x,y).
\end{eqnarray*}
The term $\bar{K}_{2,n}(x,y)$ can be treated as $\tilde{K}_{2,n}$ above. Hence it remains to deal with $\check{K}_{2,n}(x,y)$.

We define
$$D_1=\{(x,y)\in D_{\O}: u_n(x)\le \tilde{k}, u_n(y) \le \tilde{k} \},$$ where $\tilde{k}>>k+||v||_\infty$ is a large constant. Using the fact that $T_{\tilde{k}}(u_n)\to T_{\tilde{k}}(u)$ strongly in $W^{s,p}_{\beta, 0}(\Omega)$, we obtain that
$$
\check{K}_{2,n}(x,y)\chi_{D_1}\to C_2|v(x)-v(y)||u(x)-u(y)|^{p-2}|T_k(w(x))-T_k(w(y))|\chi_{\{u(x)\le \tilde{k}, u(y)\le \tilde{k}\}}$$
strongly  in  $L^1(D_{\O},d\nu)$.

\

Now, consider the set
$$D_2=\{(x,y)\in D_{\O}: u_n(x)\ge k_1, u_n(y) \ge k_1 \},$$ where $k_1>k+||v||_\infty$, then $\check{K}_{2,n}(x,y)\chi_{D_2}(x,y)=0$.

\

Hence we just have to deal with the set of the form
$$D_3=\{(x,y)\in D_{\O}: u_n(x)\ge 2k, u_n(y) \le k \},$$ or
$$D_4=\{(x,y)\in D_{\O}: u_n(y)\ge 2k, u_n(x) \le k \}.$$

We will use Remark \ref{RR} and a duality argument.

It is clear that for  $(x,y)\in D_3$, we have
$$
\check{K}_{2,n}(x,y)\chi_{D_3}(x,y)\le C(k)|v(x)-v(y)||T_k(w_n(x))-T_k(w_n(y))|u^{p-2}_n(x)\chi_{D_3}(x,y).
$$
From Remark \ref{RR}, we know that
$$
u^{p-2}_n(x)\chi_{D_3}(x,y)\rightharpoonup u^{p-2}(x)\chi_{\{u(x)\ge 2k, u(y)\le k\}} \mbox{  weakly  in   }L^{\frac{p-1}{p-2}}(D_\O, d\nu).
$$
Since
\begin{equation*}
\begin{split}
&\Big[|v(x)-v(y)||T_k(w_n(x)-T_k(w_n(y)|\Big]^{p-1}\le \frac{p-1}{p}|T_k(w_n(x))-T_k(w_n(y))|^p+\frac{1}{p}|v(x)-v(y)|^{p(p-1)}\\
&\le \frac{p-1}{p}|T_k(w_n(x))-T_k(w_n(y))|^p+\frac{1}{p}(2||v||_\infty)^{p(p-2)}|v(x)-v(y)|^{p}\\
&=: L_n(x,y).
\end{split}
\end{equation*}
Clearly $L_n\to L  \mbox{ strongly   in   }L^1(D_\O, d\nu)$ with
$$
L(x,y)=\frac{p-1}{p}|T_k(w(x))-T_k(w(y))|^p+\frac{1}{p}(2||v||_\infty)^{p(p-2)}|v(x)-v(y)|^{p}.
$$
Thus
$$
\check{K}_{2,n}\chi_{D_3}\to C_2|v(x)-v(y)||u(x)-u(y)|^{p-2}|T_k(w(x))-T_k(w(y))|\chi_{\{u(x)\ge 2k, u(y)\le k\}}$$
strongly  in $L^1(D_{\O},d\nu)$.

\

In the same way we can treat the set $D_4$.

Therefore, combining the above estimates and using the Dominate Convergence theorem, we conclude that
\begin{eqnarray*}
&\dyle \iint_{D_{\O}}K_{2,n}(x,y)d\nu\to \\
& \dyle \iint_{D_{\O}}\Big[U(x,y)-|(u(x)-v(x))-(u(y)-v(y))|^{p-2}((u(x)-v(x))-(u(y)-v(y)))\Big]\\ &\times [T_k(u(x)-v(x))-T_k(u(y)-v(y))]d\nu
\end{eqnarray*}
as $n\to \infty$ and the claim follows.

Hence, as a conclusion we obtain that
\begin{eqnarray*}
&\dyle \frac 12\iint_{D_\O}
\,U(x,y)[T_k(u(x)-v(x))-T_k(u(y)-v(y))]d\nu\le\\
&\dyle \io f(x)T_k(u(x)-v(x)) \, dx
\end{eqnarray*}
and the result follows at once. \cqd

It is clear that if $u$ is an entropy solution of \eqref{main00}, then for all $w\in C^\infty_0(\Omega)$, we have
\begin{equation}\label{OOP}
\dyle \frac 12\iint_{D_\O}
\,U(x,y)(w(x)-w(y))d\nu= \io f(x)w(x)\, dx,
\end{equation}
where $U(x,y)=|u(x)-u(y)|^{p-2}(u(x)-u(y))$.

Moreover, we can prove that \eqref{OOP} holds for all $w\in W^{s,p}_{\beta, 0}(\O) \cap L^{\infty}(\O)$ such that $w\equiv 0$ in the set $\{u>k\}$ for some $k>0$. More precisely we have
\begin{Lemma}\label{ness}
Assume that $u$ is an entropy solution to \eqref{main00} with $f\gneq 0$, then for all $w\in W^{s,p}_{\beta, 0}(\O) \cap L^{\infty}(\O)$ such that, for some $k>0$,  $w\equiv 0$ in the set $\{u>k\}$,  we have
$$
\dyle \frac 12\iint_{D_\O}
\,U(x,y)(w(x)-w(y))d\nu= \io f(x)w(x)\, dx.
$$
\end{Lemma}
\begin{proof}
Let $w\in W^{s,p}_{\beta, 0}(\O) \cap L^{\infty}(\O)$ be such that $w\equiv 0$ in the set $\{u>k_0\}$ for some $k_0>0$ and define
$v_h=T_h(u-w)$ with $h>>k_0+||w||_\infty+1$.

Since $u$ is an entropy solution to \eqref{main00}, then for $k$ fixed such that $k>>\max\{k_0, ||w||_\infty\}$, we have
\begin{eqnarray*}
&\dyle \frac 12\iint_{D_\O}
\,U(x,y)[T_k(u(x)-v_h(x))-T_k(u(y)-v_h(y))]d\nu\le\\
&\dyle \io f(x)T_k(u(x)-v_h(x)) \, dx.
\end{eqnarray*}
It is clear that
$$
\io f(x)T_k(u(x)-v_h(x)) \, dx\to \io f(x)w\, dx\mbox{  as  }h\to \infty.
$$
Notice that, for $h>>||w||_\infty$, we have $\bigg\{u\le w-h\bigg\}=\emptyset$, thus for $h$ as above there results that
$$
\bigg\{|u(x)-w(x)|\ge h\bigg\}\equiv \bigg\{(u(x)-w(x))\ge h\bigg\}.$$
Define
$$
A_h\equiv\bigg\{(x,y)\in D_\O: |u(x)-w(x)|<h,\: |u(y)-w(y)|<h\bigg\},$$
\begin{eqnarray*}
B_h &\equiv & \bigg\{(x,y)\in D_\O: |u(x)-w(x)|\ge h,\:\: |u(y)-w(y)|\ge h\bigg\}\\
&=& \bigg\{(x,y)\in D_\O: (u(x)-w(x))\ge h,\:\: (u(y)-w(y))\ge h\bigg\}
\end{eqnarray*}
and
$$
E_h\equiv\bigg\{(x,y)\in D_\O: (u(x)-w(x))\ge h,\: |u(y)-w(y)|\le h\bigg\},
$$
$$
F_h\equiv\bigg\{(x,y)\in D_\O: |u(x)-w(x)|< h,\: (u(y)-w(y))> h\bigg\}.
$$
Then
\begin{equation*}
\begin{split}
&\dyle \iint_{D_\O}
\,U(x,y)[T_k(u(x)-v_h(x))-T_k(u(y)-v_h(y))]d\nu=\iint_{A_h}+ \iint_{B_h}+\iint_{E_h}+\iint_{F_h}\\
&= I_{A_h}+I_{B_h}+I_{E_h}+I_{F_h}.
\end{split}\end{equation*}
It is clear that
\begin{equation*}
\begin{split}
& I_{A_h}=\dyle \iint_{A}
\,U(x,y)[T_k(w(x))-T_k(w(y)]d\nu=\iint_{A}
\,U(x,y)[w(x)-w(y)]d\nu\\
&= \dyle \iint_{A\cap \{u(x)<k_0, u(y)<k_0\}}
\,U(x,y)[w(x))-w(y)]d\nu+\iint_{A\cap\{u(x)>k_0, u(y)>k_0\}}
\,U(x,y)[w(x)-w(y)]d\nu\\ &+\dyle \iint_{A\cap\{u(x)>k_0, u(y)\le k_0\}}\,U(x,y)[w(x)-w(y)]d\nu
+\iint_{A\cap\{u(x)\le k_0, u(y)>k_0\}}\,U(x,y)[w(x)-w(y)]d\nu\\ \\
&= I_{1}(h)+I_{2}(h)+I_{3}(h)+I_{4}(h)
\end{split}
\end{equation*}
Since $T_k(u)\in W^{s,p}_{\beta, 0}(\O)$, we have
$$
I_{1}(h)\to \iint_{\{u(x)<k_0, u(y)<k_0\}}
\,U(x,y)[w(x))-w(y)]d\nu\mbox{  as  }h\to \infty.
$$
Using the properties of $w$, we have $I_{2}(h)=0$. Let us consider now $I_{3}(h)$, we have
\begin{equation*}
\begin{split}
I_{3}(h) &=\dyle \iint_{A\cap\{k_0<u(x)<2k_0, u(y)\le k_0\}}\,U(x,y)[w(x)-w(y)]d\nu\\
& +\dyle \iint_{A\cap\{u(x)>2k_0, u(y)\le k_0\}}\,U(x,y)[w(x)-w(y)]d\nu\\
&=J_1(h)+J_2(h).
\end{split}
\end{equation*}
As above, since $T_k(u)\in W^{s,p}_{\beta, 0}(\O)$, then
$$
J_{1}(h)\to \iint_{\{k_0<u(x)<2k_0, u(y)<k_0\}}
\,U(x,y)[w(x))-w(y)]d\nu\mbox{  as  }h\to \infty.
$$
For $J_2(h)$, we have
$$
\Big|U(x,y)[w(x))-w(y)]\Big|\le ||w||_\infty\Big|U(x,y)\Big|=||w||_\infty\Big|u(x)-u(y)\Big|^{p-1}.
$$
Using the fact that $$
\iint_{\{u(x)>2k_0, u(y)\le k_0\}}\,\Big|U(x,y)\Big|d\nu<\infty,
$$
then by the dominated convergence theorem we conclude that
$$
J_{2}(h)\to \iint_{\{u(x)\ge 2k_0, u(y)<k_0\}}
\,U(x,y)[w(x))-w(y)]d\nu\mbox{  as  }h\to \infty.
$$
In the same way we can treat $I_{4}(h)$. Hence
$$
I_{A_h}\to \iint_{D_\O}
\,U(x,y)[w(x)-w(y)]d\nu\mbox{  as  }h\to \infty.
$$
We deal now with $I_{B_h}$. It is clear that if $(x,y)\in B_h$, then $v_h(x)=v_h(y)=h$, hence
$$
I_{B_h}=\iint_{B_h}
\,U(x,y)[T_k(u(x)-h)-T_k(u(y)-h)]d\nu\ge 0.
$$
Now, for $(x,y)\in E_h$, we have $u(x)\ge h-||w||_\infty>k_0$, thus $w(x)=0$. Hence
$$
E_h\equiv E_h\cap \{u(y)<h-||w||_\infty-1\}\cup E_h\cap \{ h\ge u(x), u(y)\ge h-||w||_\infty-1\}\equiv E_1(h)\cup E_2(h).
$$
It is clear that for $(x,y)\in E_2(h)$, we have $w(x)=w(y)=0$, then
$$
\,U(x,y)[T_k(u(x)-v(x))-T_k(u(y)-v_h(y))]=\,U(x,y)[T_k(G_h(u(x)))-T_k(G_h(u(y)))]\ge 0.
$$
Thus
$$
\iint_{E_2(h)}\,U(x,y)[T_k(u(x)-v(x))-T_k(u(y)-v_h(y))]d\nu\ge 0.
$$
Therefore, we conclude that
\begin{equation*}
I_{E_h}\ge \dyle \iint_{E_1(h)}
\,U(x,y)[T_k(G_h(u(x)))-T_k(w(y))]d\nu\ge -2k\iint_{E_1(h)}
\,\Big|U(x,y)\Big|d\nu.
\end{equation*}
Let $h_1=h-||w||_\infty-1$, then
$$
\iint_{E_1(h)}
\,\Big|U(x,y)\Big|d\nu\le \iint_{u(x)>h_1, u(y)<h_1-1}
\,\Big|U(x,y)\Big|d\nu\to 0\mbox{  as   }h\to \infty.
$$
Thus $I_{E_h}\ge o(h)$.

In the same way we can prove that $I_{F_h}\ge o(h).$
Therefore we reach that
$$
\liminf_{h\to \infty}\dyle \frac 12\iint_{D_\O}
\,U(x,y)[T_k(u(x)-v(x))-T_k(u(y)-v_h(y))]d\nu\ge
\frac 12\iint_{D_\O}
\,U(x,y)(w(x)-w(y))d\nu
$$
As a conclusion we have proved that
$$
\frac 12\iint_{D_\O}
\,U(x,y)(w(x)-w(y))d\nu\le \io f(x)w(x)dx.
$$
Substituting $w$ by $-w$ in the above inequality, we obtain that
$$
\frac 12\iint_{D_\O}
\,U(x,y)(w(x)-w(y))d\nu=\io f(x)w(x)dx
$$
which is the desired result.
\end{proof}
Now, we are in position to prove the uniqueness result in Theorem \ref{entropi}.

{\bf{Proof of Theorem \ref{entropi}: Uniqueness part:}}

Let $u$ be the entropy positive solution defined in Theorem \ref{mainth}, recall that $u=\limsup u_n$ where $u_n$ is the unique solution to the approximated problem \eqref{pro:lineal1}.

Assume that $v$ is an other entropy positive solution to problem \eqref{main00}. We claim that $u_n\le v$ for all $n$.
To prove the claim we fix $n$ and define $w_n=(u_n-v)_+$, then $w_n=(u_n-T_k(v))_+$ where $k>>||u_n||_\infty$. Hence $w_n\in W^{s,p}_{\beta, 0}(\O) \cap L^{\infty}(\O)$ and $w_n\equiv 0$ in the set $\{v>||u_n||_\infty\}$. Therefore using $w_n$ as a test function in \eqref{pro:lineal1} and taking into consideration the result of Lemma \ref{ness}, we reach that
\begin{eqnarray*}
\dyle \frac 12\iint_{D_\O}
\,U_n(x,y)(w_n(x)-w_n(y))d\nu &= &\io f_n(x)w_n(x)\, dx\\
&\le &\dyle \io f(x)w_n(x)\, dx=\dyle \frac 12\iint_{D_\O}
\,V(x,y)(w_n(x)-w_n(y))d\nu,
\end{eqnarray*}
where
$$
U_n(x,y)=|u_n(x)-u_n(y)|^{p-2}(u_n(x)-u_n(y))\mbox{   and   }V(x,y)=|v(x)-v(y)|^{p-2}(v(x)-v(y)).
$$
Thus
$$
\dyle \frac 12\iint_{D_\O}
\,\Big(U_n(x,y)-V(x,y)\Big)(w_n(x)-w_n(y))d\nu
\le 0.
$$
Using the fact that
$$
\Big(U_n(x,y)-V(x,y)\Big)(w_n(x)-w_n(y))\ge C|w_n(x)-w_n(y)|^p,
$$
it follows that $w_n\equiv 0$, hence $u_n\le v$ for all $n$ and the claim follows.
As a consequence we reach that $u\le v$.

\

\

Let us prove now that $v\le u$. To this aim, we will follow closely the argument used in \cite{BZ}.

Since $u, v$ are entropy solutions to \eqref{main00}, then for $h>>k$, we have
\begin{eqnarray*}
&\dyle \frac 12\iint_{D_\O}
\,U(x,y)[T_k(u(x)-T_h(v(x)))-T_k(u(y)-T_h(v(y)))]d\nu\le\dyle \io f(x)T_k(u(x)-T_h(v(x))) \, dx,
\end{eqnarray*}
and
\begin{eqnarray*}
&\dyle \frac 12\iint_{D_\O}
\,V(x,y)[T_k(v(x)-T_h(u(x)))-T_k(v(y)-T_h(u(y)))]d\nu\le\dyle \io f(x)T_k(v(x)-T_h(u(x))) \, dx.
\end{eqnarray*}
It is clear that
$$
\io f(x)T_k(u(x)-T_h(v(x))) \, dx+\io f(x)T_k(v(x)-T_h(u(x))) \, dx\to 0\mbox{  as  }h\to \infty.
$$
Thus
\begin{equation}\label{santi}
\begin{array}{lll}
I(h)&\equiv & \dyle \frac 12\iint_{D_\O}
\,U(x,y)[T_k(u(x)-T_h(v(x)))-T_k(u(y)-T_h(v(y)))]d\nu\\ &+ & \dyle \frac 12\iint_{D_\O}
\,V(x,y)[T_k(v(x)-T_h(u(x)))-T_k(v(y)-T_h(u(y)))]d\nu\\
&=& P(h)+Q(h)\le o(h).
\end{array}
\end{equation}
Let
$$
D^1_{\O}(h)\equiv \{(x,y)\in D_\O\mbox{  such that  } u(x)<h\mbox{  and  }u(y)<h\}
$$
and
$$
D^2_{\O}(h)\equiv \{(x,y)\in D_\O\mbox{  such that  } v(x)<h\mbox{  and  }v(y)<h\}.
$$
Then
\begin{equation*}
\begin{array}{lll}
P(h) & = & \dyle \iint_{D^1_\O(h)}
\,U(x,y)[T_k(u(x)-T_h(v(x)))-T_k(u(y)-T_h(v(y)))]d\nu \\
&+ & \dyle \iint_{D_\O\backslash D^1_\O(h)}
\,U(x,y)[T_k(u(x)-T_h(v(x)))-T_k(u(y)-T_h(v(y)))]d\nu\\ \\
&=& P_1(h)+P_2(h),
\end{array}
\end{equation*}
and
\begin{equation*}
\begin{array}{lll}
Q(h) & =& \dyle\iint_{D^2_\O(h)}
\,V(x,y)[T_k(v(x)-T_h(u(x)))-T_k(v(y)-T_h(u(y)))]d\nu\\
&+ & \dyle \iint_{D_\O\backslash D^2_\O(h)}
\,V(x,y)[T_k(v(x)-T_h(u(x)))-T_k(v(y)-T_h(u(y)))]d\nu\\ \\
&=& Q_1(h)+Q_2(h).
\end{array}
\end{equation*}

\

We claim that $P_2(h)\ge o(h)$ and $Q_2(h)\ge o(h)$.

Let us begin by proving that $P_2(h)\ge o(h)$. It is clear that
$$
D_\O\backslash D^1_\O(h)=\{(x,y)\in D_\O\mbox{  with  }u(x)\ge h\}\cup \{(x,y)\in D_\O\mbox{  with  }u(y)\ge h\}.
$$
If $u(x)\ge h$ and $u(y)\ge h$, then $v(x)\ge h$ and $v(y)\ge h$. Thus
$$
U(x,y)[T_k(u(x)-T_h(v(x)))-T_k(u(y)-T_h(v(y)))]=U(x,y)[T_k(u(x)-h)-T_k(u(y)-h)]\ge 0.
$$
On the other hand, by \eqref{entro001}, we get
$$
\iint_{\{u(x)>h, u(y)<h-1\}}|U(x,y)|d\nu=o(h).
$$
Hence
\begin{equation*}
\begin{array}{lll}
&\dyle \iint_{\{u(x)\ge h\}}
\,U(x,y)[T_k(u(x)-T_h(v(x)))-T_k(u(y)-T_h(v(y)))]d\nu\\ \\
&\ge \dyle \iint_{\{u(x)\ge h, h-1\le u(y)\le h\}}
\,U(x,y)[T_k(u(x)-h)-T_k(u(y)-T_h(v(y)))]d\nu+o(h).
\end{array}
\end{equation*}
Notice that for $(x,y)\in \{u(x)\ge h, h-1\le u(y)\le h\}$, we have
$$
U(x,y)[T_k(u(x)-h)-T_k(u(y)-T_h(v(y)))]=U(x,y)[T_k(u(x)-h)+T_k(T_h(v(y))-u(y))]\ge 0.
$$
Thus
$$
\dyle \iint_{\{u(x)\ge h\}}
\,U(x,y)[T_k(u(x)-T_h(v(x)))-T_k(u(y)-T_h(v(y)))]d\nu\ge o(h).
$$
In the same way we can prove that
$$
\dyle \iint_{\{u(y)\ge h\}}
\,U(x,y)[T_k(u(x)-T_h(v(x)))-T_k(u(y)-T_h(v(y)))]d\nu\ge o(h).
$$
Thus $P_2(h)\ge o(h)$ as affirmed.

\

We deal now with $Q_2(h)$. As above, we have
$$
D_\O\backslash D^2_\O(h)=\bigg\{(x,y)\in D_\O \mid: v(x)\ge h\}\cup \{(x,y)\in D_\O \mid v(y)\ge h\bigg\}\equiv M_1(h)\cup M_2(h)\cup M_3(h),
$$
where
$$
M_1(h)=\bigg\{(x,y)\in D_\O\mid: v(x)\ge h\mbox{  and   }v(y)\ge h\bigg\},
M_2(h)=\bigg\{(x,y)\in D_\O\mid: v(x)\ge h\mbox{  and   }v(y)< h\bigg\},
$$
and
$$
M_3(h)=\bigg\{(x,y)\in D_\O\mid: v(x)<h\mbox{  and   }v(y)\ge h\bigg\}.
$$
Let
$$
Z_1(h)=\bigg\{(x,y)\in D_\O\mid: v(x)-T_h(u(x))\ge k\bigg\} \mbox{  and } Z_2(h)=\bigg\{(x,y)\in D_\O\mid: v(y)-T_h(u(y))\ge k\bigg\}.
$$
If $(x,y)\in Z_1(h)\cap Z_2(h)$, we have
$$
\,V(x,y)[T_k(v(x)-T_h(u(x)))-T_k(v(y)-T_h(u(y)))]=0.
$$
Hence we can assume that $(x,y)\in \Big(Z_1(h)\backslash Z_2(h)\Big)\cup \Big(Z_2(h)\backslash Z_1(h)\Big)\equiv Y_1(h)\cup Y_2(h)$.

Therefore we conclude that
\begin{equation*}
\begin{array}{lll}
Q_2(h) & =& \dyle\iint_{M_1(h)\cap Y_1(h)}+\iint_{M_2(h)\cap Y_1(h)}+\iint_{M_3(h)\cap Y_1(h)}\\
&+ & \dyle \iint_{M_1(h)\cap Y_2(h)}+\iint_{M_2(h)\cap Y_2(h)}+\iint_{M_3(h)\cap Y_2(h)}\\ \\
&=& J_1(h)+J_2(h)+J_3(h)+T_1(h)+T_2(h)+T_3(h).
\end{array}
\end{equation*}
Let us begin by proving that $J_1(h)\ge o(h)$. Notice that
$$
M_1(h)\cap Y_1(h)=\bigg\{
(x,y)\in D_\O\mid: v(x)\ge h, v(y)\ge h\mbox{  and  } v(x)-T_h(u(x))\ge k, v(y)-T_h(u(y))<k\bigg\},
$$
then for $(x,y)\in M_1(h)\cap Y_1(h)$, we have
$$
\,V(x,y)[T_k(v(x)-T_h(u(x)))-T_k(v(y)-T_h(u(y)))]=\,V(x,y)[k-(v(y)-T_h(u(y)))]
$$
If $v(x)\ge v(y)$, then $\,V(x,y)[k-(v(y)-T_h(u(y)))]\ge 0$, so we have just to consider the case where $(x,y)\in M_1(h)\cap Y_1(h)$ with $v(x)<v(y)$. Thus
$$
\Big|\,V(x,y)[T_k(v(x)-T_h(u(x)))-T_k(v(y)-T_h(u(y)))]\Big|=\,(v(y)-v(x))^{p-1}[k-(v(y)-T_h(u(y)))].
$$
Now taking into consideration that $(x,y)\in M_1(h)\cap Y_1(h)$, we get
$$
0\le (v(y)-v(x))\le T_h(u(y))+k-(T_h(u(x))+k)\le T_h(u(y))-T_h(u(x))\le u(y)-u(x).
$$
Therefore we conclude that
$$
\Big|\,V(x,y)[T_k(v(x)-T_h(u(x)))-T_k(v(y)-T_h(u(y)))]\Big|\le 2k\,(u(y)-u(x))^{p-1}.
$$
If $u(x)\ge h$, then $u(y)\ge h$, hence
$$
\,V(x,y)[T_k(v(x)-T_h(u(x)))-T_k(v(y)-T_h(u(y)))]=\,V(x,y)[T_k(v(x))-T_k(v(y))]\ge 0,
$$
It then remains to consider the case $u(x)<h$.
\begin{enumerate}

\item[(i)] If $u(y)>(h+1)$, by \eqref{entro001}, we reach that
$$
\iint_{\{u(y)>h+1, u(x)<h\}}|U(x,y)|d\nu=o(h).
$$
\item [(ii)]If $h<u(y)\le (h+1)$, then $0\le k-(v(y)-T_h(u(y)))\le u(y)-u(x)$, thus
$$
\Big|\,V(x,y)[T_k(v(x)-T_h(u(x)))-T_k(v(y)-T_h(u(y)))]\Big|\le\,(u(y)-u(x))^{p}.
$$
Now, by \eqref{two21}, we get
$$
\iint_{\{h<u(y)\le h+1, u(x)<h\}}(u(y)-u(x))^p d\nu=o(h).
$$
\item[(iii)] We deal now with the set $u(y)\le h$.  Since $(x,y)\in M_1(h)\cap Y_1(h)$, we have
$u(y)\ge (h-k)$, so if $u(x)<(h-k-1)$, using again \eqref{two2} we reach that
$$
\iint_{\{u(y)>(h-k), u(x)<(h-k-1)\}}|U(x,y)|d\nu=o(h).
$$
Let us assume that $(h-k-1)<u(x)\le u(y)<h$. In this case we have
$$
[k-(v(y)-T_h(u(y)))]=u(y)-(v(y)-k)\le u(y)-(v(x)-k)\le u(y)-u(x).
$$
So for $(x,y)\in M_1(h)\cap Y_1(h)$ with $h-k-1<u(x)\le u(y)<h$, we get
$$
\Big|\,V(x,y)[T_k(v(x)-T_h(u(x)))-T_k(v(y)-T_h(u(y)))]\Big|=\,(v(y)-v(x))^{p-1}[k-(v(y)-T_h(u(y)))]\le (u(y)-u(x))^{p}.
$$
Now using \eqref{two21}, it follows that
$$
\iint_{\{h-k-1<u(x)\le u(y)<h\}}(u(y)-u(x))^p d\nu=o(h).
$$
\end{enumerate}

Therefore combining the above estimates we obtain $J_1(h)\ge o(h)$.

\

For $J_2(h)$, we have $v(y)\le h<v(x)$ and
$$
T_k(v(x)-T_h(u(x)))-T_k(v(y)-T_h(u(y)))\ge k-(v(y)-T_h(u(u)))\ge 0.
$$
Thus
$$
\,V(x,y)[T_k(v(x)-T_h(u(x)))-T_k(v(y)-T_h(u(y)))]\ge 0\mbox{  for all }(x,y)\in M_2(h)\cap Y_1(h).
$$
Hence $J_2(h)\ge 0$.

We deal now with $J_3(h)$. We have $v(x)\le h<v(y)$ and for all $(x,y)\in M_3(h)\cap Y_1(h)$,
$$
\,\Big|V(x,y)[T_k(v(x)-T_h(u(x)))-T_k(v(y)-T_h(u(y)))]\Big|=\,(v(y)-v(x))^{p-1}[k-(v(y)-T_h(u(y)))].
$$
If $v(x)\le (h-1)$, then by \eqref{two2} we have
\begin{eqnarray*}
& \dyle \iint_{\{v(y)>h, v(x)<h-1\}}\,\Big|V(x,y)[T_k(v(x)-T_h(u(x)))-T_k(v(y)-T_h(u(y)))]\Big|d\nu\le \\ \\
& \dyle2k\iint_{\{v(y)>h, v(x)<h-1\}}|V(x,y)|d\nu=o(h).
\end{eqnarray*}
So, assume that $(h-1)<v(x)\le h$. Since $(x,y)\in M_3(h)\cap Y_1(h)$, then $v(y)\le k+T_h(u(y))\le k+u(y)$. Thus $u(y)>(h-k)$.
It is clear that
$$
0\le v(y)-v(x)\le T_h(u(y))-T_h(u(x))\le u(y)-u(x).
$$
Hence following the same discussion as in case $(iii)$ in the analysis of $J_1(h)$, and combining the above estimates we reach that $J_3(h)\ge 0$.

\

Notice that in a symmetric way we can prove that $T_1(h)+T_2(h)+T_3(h)\ge o(h)$.
Thus $Q_2(h)\ge o(h)$ and the claim follows.

\

Therefore, going back to the definition of $I(h)$ given in \eqref{santi} and taking into consideration that $u<h$ in the set $\{v<h\}$, it follows that
\begin{eqnarray*}
I(h)&\ge & \dyle \frac 12\iint_{D^1_{\O}(h)}
\,U(x,y)[T_k(u(x)-T_h(v(x)))-T_k(u(y)-T_h(v(y)))]d\nu\\ &+ & \dyle \frac 12\iint_{D^2_{\O}(h)}
\,V(x,y)[T_k(v(x)-u(x))-T_k(v(y)-u(y))]d\nu+o(h)\\
&\ge & \frac 12\iint_{D^2_{\O}(h)}
\,\Big(V(x,y)-U(x,y)\Big)[T_k(v(x)-u(x))-T_k(v(y)-u(y))]d\nu\\
&+& \dyle \frac 12\iint_{D^1_{\O}(h)\backslash D^2_{\O}(h)}
\,U(x,y)[T_k(u(x)-T_h(v(x)))-T_k(u(y)-T_h(v(y)))]d\nu +o(h)\\
&\ge & I_1(h)+I_2(h)+o(h).
\end{eqnarray*}
It is clear that
$$
I_1(h)\ge C\iint_{D^2_{\O}(h)}
\,|T_k(v(x)-u(x))-T_k(v(y)-u(y))|^pd\nu.
$$
We claim that $I_2(h)\ge o(h)$.

Notice that
$$
D^1_{\O}(h)\backslash D^2_{\O}(h)=N_1(h)\cup N_2(h)\cup N_3(h)
$$
where
$$
N_1(h)\equiv \bigg\{(x,y)\in D_\O \mbox{  such that  }u(x)\le h, u(y)\le h, v(x)>h, v(y)>h\bigg\},$$
$$
N_2(h)\equiv \bigg\{(x,y)\in D_\O \mbox{  such that  }u(x)\le h, u(y)\le h, v(x)>h, v(y)\le h\bigg\},$$
and $$
N_3(h)\equiv \bigg\{(x,y)\in D_\O \mbox{  such that  }u(x)\le h, u(y)\le h, v(x)\le h, v(y)>h\bigg\}.
$$
It is clear that
$$
\frac 12\iint_{N_1(h)}
\,U(x,y)[T_k(u(x)-T_h(v(x)))-T_k(u(y)-T_h(v(y)))]d\nu\ge 0.
$$
Therefore, we conclude that
$$
I_2(h)\ge \frac 12\iint_{N_2(h)} + \frac 12\iint_{N_3(h)}=I_{21}(h)+I_{22}(h).
$$

\

\

For $(x,y)\in N_2(h)$, we will consider three main cases:
\begin{itemize}

\item[$I)$] If $h-u(x)\le v(y)-u(y)$, then $0\le h-v(y)\le u(x)-u(y)$. Hence
$$
\,U(x,y)[T_k(u(x)-T_h(v(x)))-T_k(u(y)-T_h(v(y)))]=\,U(x,y)[T_k(v(y)-u(y))-T_k(h-u(x)]\ge 0.
$$
\

\item[$II)$] If $u(x)-u(y)\le 0\le h-v(y)$, then $u(x)-u(y)\le 0$ and $h-u(x)\ge v(y)-u(y)$. Thus
$$
\,U(x,y)[T_k(u(x)-T_h(v(x)))-T_k(u(y)-T_h(v(y)))]=\,U(x,y)[T_k(v(y)-u(y))-T_k(h-u(x)]\ge 0.
$$
\

\item[$III)$] Consider now the case where $0\le u(x)-u(y)\le h-v(y)$. It is clear that $0\le u(x)-u(y)\le v(x)-v(y)$.
Hence
\begin{eqnarray*}
&\Big|\,U(x,y)[T_k(u(x)-T_h(v(x)))-T_k(u(y)-T_h(v(y)))]\Big|\\
&=(u(x)-u(y))^{p-1}[T_k(h-u(x))-T_k(v(y)-u(y))]\le 2k\,(v(x)-v(y))^{p-1}.
\end{eqnarray*}
If $v(y)\le h-1$ or $v(x)\ge h+1$, by \eqref{entro001}, we get
$$
\iint_{N_2(h)\cap\Big\{\{v(x)>h, v(y)<h-1\}\cup \{v(x)>h+1, v(y)<h\}\Big\}}|V(x,y)|d\nu=o(h).
$$
Thus, we deal with the set $\{ h-1<v(y)\le h \mbox{  and } v(x)\le h+1\}$.

It is clear that if $v(y)-u(y)\ge k$, then $h-u(x)\ge k$. Thus
$$
\Big|\,U(x,y)[T_k(u(x)-T_h(v(x)))-T_k(u(y)-T_h(v(y)))]\Big|=0.
$$
Assume that $h-u(x)\le k$, then $v(y)-u(y)\le k$, hence
\begin{eqnarray*}
&\Big|\,U(x,y)[T_k(u(x)-T_h(v(x)))-T_k(u(y)-T_h(v(y)))]\Big|\le \\
& (v(x)-v(y))^{p-1}[(h-u(x))-(v(y)-u(y))]\le\,(v(x)-v(y))^{p}.
\end{eqnarray*}
Therefore, using \eqref{two21},
\begin{eqnarray*}
&\dyle \iint\limits_{N_2(h)\cap\{h-1<v(y)\le h\le v(x)\le h+1\} \cap \{h-k\le u(x)\}}\Big|\,U(x,y)[T_k(u(x)-T_h(v(x)))-T_k(u(y)-T_h(v(y)))]\Big|d\nu\le \\ \\
& \dyle\iint\limits_{N_2(h)\cap\{h-1<v(y)\le h\le v(x)\le h+1\}}\,(v(x)-v(y))^{p}d\nu=o(h).
\end{eqnarray*}
We consider now the set where $v(y)-u(y)<k<h-u(x)$, then $u(x)<h-k$ and thus $u(y)<h-k$. As above we have
\begin{eqnarray*}
&\Big|\,U(x,y)[T_k(u(x)-T_h(v(x)))-T_k(u(y)-T_h(v(y)))]\Big|\le \\
& (v(x)-v(y))^{p-1}[(k-(v(y)-u(y))]\le (v(x)-v(y))^{p}
\end{eqnarray*}
Thus using again \eqref{two21},
\begin{eqnarray*}
&\dyle \iint\limits_{N_2(h)\cap\{h-1<v(y)\le h\le v(x)\le h+1\} \cap \{u(x)<h-k\}}\Big|\,U(x,y)[T_k(u(x)-T_h(v(x)))-T_k(u(y)-T_h(v(y)))]\Big|d\nu\le \\ \\
& \dyle\iint\limits_{N_2(h)\cap\{h-1<v(y)\le h\le v(x)\le h+1\}}\,(v(x)-v(y))^{p}d\nu=o(h).
\end{eqnarray*}
\end{itemize}

Therefore we conclude that $I_{21}(h)\ge o(h)$. In the same way and using a symmetric argument, we can prove that $I_{22}(h)\ge o(h)$.

Hence $I_2(h)\ge o(h)$ and the claim follows.

In conclusion, we have proved that
$$
C\iint_{D^2_{\O}(h)}
\,|T_k(v(x)-u(x))-T_k(v(y)-u(y))|^pd\nu\le o(h).
$$
Letting $h\to \infty$, there results that
$$
\iint_{D_\O}
\,|T_k(v(x)-u(x))-T_k(v(y)-u(y))|^pd\nu=0.
$$
Thus $T_k(u)=T_k(v)$ for all $k$ and then $u=v$.
\cqd

\section{Problem with reaction term and general datum}\label{sec3}

In this section we consider the problem

\begin{equation}\label{eq:exis2}
\left\{
\begin{array}{rcll}
(-\D)^s_{p,\beta} u &= & \l u^{q}+g(x) &\mbox{ in }\O,\\
u &\ge & 0 &\mbox{  in   } \O,\\
u &= & 0 &\mbox {  in  }\ren\backslash \O,
\end{array}
\right.
\end{equation}
where $\l, q>0$ and $g\gneq 0$. According to the values of $q$ and $\l$, we will prove that  problem \eqref{eq:exis2} has an entropy solution in the sense of Definition \ref{def:entropy}.

Let begin by the case $q<p-1$. We have the next existence result.
\begin{Theorem}\label{exisqq}
Assume that $q<p-1$, then for all $g\in L^1(\O)$ and for all $\l>0$, problem \eqref{eq:exis2} has a positive entropy solution.
\end{Theorem}
\begin{pf}
Without loss of generality we can assume that $\l=1$. We set $g_n=T_n(g)$, then $g_n\gneq 0$ and $g_n\uparrow g$ strongly in $L^1(\Omega)$. Define $u_n$ to be the unique solution to the approximated problem
\begin{equation}\label{appr1}
\left\{
\begin{array}{rcll}
(-\D)^s_{p,\beta} u_n &= & u^{q}_n+g_n &\mbox{ in }\O,\\
u_n &\ge & 0 & \mbox{  in   } \O,\\
u_n &= & 0 & \mbox {  in  }\ren\backslash \O.
\end{array}
\right.
\end{equation}
Notice that the existence of $u_n$ can be obtained as a critical point of the functional
$$
J(u)=\frac{1}{2p}\dyle\iint_{D_{\O}} \,|u(x)-u(y)|^{p}d\nu-\frac{1}{q+1}\io u^{q+1}_+dx-\io g_n\: u\:dx.
$$
However the uniqueness follows using the comparison result in Lemma \ref{compa}. It is clear that by the same comparison principle we obtain that $u_n\le u_{n+1}$.

We claim that $\{u^{p-1}_n\}_n$ is uniformly bounded in $L^1(\O)$. To prove the claim we argue by contradiction. Assume that $C_n\equiv ||u^{p-1}_n||_{L^1(\O)}\to \infty$ as $n\to \infty$. We set $v_n=\frac{u_n}{C^{\frac{1}{p-1}}_n}$, then $||v^{p-1}_n||_{L^1(\O)}=1$ and $v_n$ solves the problem
\begin{equation}\label{appr1001}
\left\{
\begin{array}{rcll}
(-\D)^s_{p,\beta} v_n &= & C^{\frac{q-p+1}{p-1}}_nv^{q}_n+C^{-1}_n g_n & \mbox{ in }\O,\\
v_n & \ge &  0  & \mbox{  in   } \O,\\
v_n & = & 0 & \mbox {  in  }\ren\backslash \O.
\end{array}
\right.
\end{equation}
We set $G_n\equiv C^{\frac{q-p+1}{p-1}}_nv^{q}_n+C^{-1}_n g_n$, then $||G_n||_{L^1(\O)}\to 0$ as $n\to \infty$.
Taking in consideration the results of Lemmas \ref{one} and \ref{two}
we get the existence of a measurable function $v$ such that $T_k(v)\in W^{s,p}_{\beta, 0}(\Omega)$, $v^{p-1}\in
L^\s(\O,|x|^{-2\beta} dx)$ for all $\s<\dfrac{N}{N-ps}$ and $T_k(v_n)\rightharpoonup T_k(v)$ weakly in $W^{s,p}_{\beta, 0}(\Omega)$.

Since $\s>1$, then using Vitali's Lemma we can prove that $v^{p-1}_n\to v^{p-1}$ strongly in $
L^1(\O)$. Thus $||v^{p-1}||_{L^1(\O)}=1$.

Now taking $T_k(v_n)$ as a test function in \eqref{appr1001}, using the fact that $||G_n||_{L^1(\O)}\to 0$, it follows that
$||T_k(v_n)||_{W^{s,p}_{\beta, 0}(\Omega)}\to 0$ as $n\to \infty$. Hence $T_k(v)=0$ for all $k$ and then $v\equiv 0$. Thus we reach a contradiction with the fact that $||v^{p-1}||_{L^1(\O)}=1$.

Therefore $||u^{p-1}_n||_{L^1(\O)}\le C$ for all $n$ and the claim follows.

Since $q<p-1$, we conclude that the sequence $\{u^q_n+g_n\}_n$ is bounded in $L^1(\O)$ and then we get the existence of a measurable function $u$ such that
$u^q_n\uparrow u^q$, $u^{p-1}\in
L^\s(\O,|x|^{-2\beta} dx)$ for all $\s<\dfrac{N}{N-ps}$ and $T_k(u_n)\rightharpoonup T_k(u)$ weakly in $W^{s,p}_{\beta, 0}(\Omega)$.

Since  $\{u^q_n+f_n\}_n$ is an increasing sequence, using Lemma \ref{compact}, we conclude that $T_k(u_n)\to T_k(u)$ strongly in $W^{s,p}_{\beta, 0}(\Omega)$.
Now, by Theorem \ref{entropi} we obtain that $u$ is an entropy solution to problem \eqref{eq:exis2}  in the sense of Definition \ref{def:entropy}.

We prove now that $u$ is the minimal solution of \eqref{eq:exis2}.

Let $\overline{u}$ be another entropy positive solution to problem \eqref{eq:exis2}. Recall that $u=\lim\limits_{n\to \infty}u_n$, so to finish we have to show that $u_n\le \overline{u}$ for all $n$. Fix $n$ and consider the sequence $\{w_{n,i}\}_i$, defined by $w_{n,0}=0$ and $w_{n,i+1}$ being the unique solution to problem
\begin{equation}\label{appr100}
\left\{
\begin{array}{rcll}
(-\D)^s_{p,\beta} w_{n,i+1} & = & w^q_{n,i}+ g_n & \mbox{ in }\O,\\
w_{n,i+1} & \ge &  0 & \mbox{  in   } \O,\\
w_{n,i+1}& = & 0 & \mbox {  in  }\ren\backslash \O.
\end{array}
\right.
\end{equation}
It is clear that the sequence $\{w_{n,i}\}_i$ is increasing in $i$ with $w_{n,i}\le u_n$ for all $i$. Hence $w_{n,i}\uparrow \bar{w}_n$ a
solution to problem \eqref{appr11}. Now by the comparison principle in Lemma \ref{compa} we conclude that $\bar{w}_n=u_n$,
and by an iteration argument we can prove that $w_{n,i}\le \overline{u}$ for all $i$. Hence $u_n\le \overline{u}$ and the result follows.
\end{pf}

In the case where $q=p-1$, the problem is related to the first eigenvalue of the operator $(-\D)^s_{p,\beta}$. More precisely, we set
\begin{equation}\label{first}
\l_1=
\inf_{\phi \in W^{s,p}_{\beta, 0}(\Omega),\,\,\phi \ne 0}\frac{\dyle\frac{1}{2}\iint_{D_{\O}} \,|\phi(x)-\phi(y)|^{p}d\nu}{\io |\phi|^p d x}.
\end{equation}
As in the case $\beta=0$, it is not difficult to show that $\l_1>0$ and that $\l_1$ is attained.

Now, we can formulate our existence result.

\begin{Theorem}\label{exisqp}
Assume that $q=p-1$. If $\l<\l_1$, then for all $g\in L^1(\O)$, problem \eqref{eq:exis2} has a minimal entropy positive solution.
\end{Theorem}

To prove Theorem \ref{exisqp}, we need the next classical regularity result.
\begin{Lemma} \label{bound}
Let $u$ be the unique solution to the problem
\begin{equation}\label{mainaa}
\left\{
\begin{array}{lll}
(-\Delta)^s_{p, \beta} u &= & f \mbox{ in }\O,\\
 u &= & 0  \mbox{ in }  \ren\setminus\O,
\end{array}
\right.
\end{equation}
where $|f||x|^{p^*_s\b}\in L^{m}(\Omega, |x|^{-p^*_s\b}\,dx)$ for some $m>\frac{N}{ps}$, then $u\in L^\infty(\O)$.
\end{Lemma}
\begin{pf}
We follow closely the Stampacchia argument given in \cite{St}. Using $G_{k}(u(x))$, with $k>0$, as a test function \eqref{mainaa},
and taking in consideration that
$$
\begin{array}{c}
U(x,y)\big(G_k (u(x))- G_k(u(y))\ge |G_k (u(x))- G_k(u(y)|^{p},
\end{array}
$$
where $U(x,y)=|u(x)-u(y)|^{p-2}(u(x)-u(y))$, we reach that
$$
\frac{1}{2}
\dint\dint_{D_\O}\frac{|G_{k}u(x)-G_{k}u(y)|^p}{|x-y|^{N+ps}}\frac{dx}{|x|^{\b}}\,
\frac{dy}{|y|^{\b}}\leq  \dint_{\Omega}{|f|}
\,|G_{k}(u(x))|\,dx.
$$
By the \textit{Weighted Sobolev Inequality} \eqref{Sobolev}, it follows that
$$
S\| G_k (u) \|_{L^{p^*_s} (\Omega, |x|^{-p^*_s\b}\,dx)}^p \leq   \int_{A_{k}}\,{|f|}\,|G_k (u(x))| dx
$$
where $A_k =\{ x \in \Omega\,: \, |u(x)|  \geq k\}$. We set $d\omega=\dfrac{dx}{|x|^{p^*_s\b}}$, then
\begin{eqnarray*}
\int_{A_{k}}\,|f|\,|G_k (u(x))| dx &= &\int_{A_{k}}\,(|f||x|^{p^*_s\b})\,|G_k (u(x))| d\omega\\
&\le & \| G_k (u) \|_{L^{p^*_s}(\O, d\omega)}^p \| (|f||x|^{p^*_s\b})\|_{L^m(\O, d\omega)}|A_k|^{1-\frac 1m-\frac{1}{p^*_s}}_{d\omega}.
\end{eqnarray*}
Thus
$$
C\| G_k (u) \|^{\frac{p-1}{p^*_s}}_{L^{p^*_s}(\O, d\omega)} \leq \| (|f||x|^{p^*_s\b})\|_{L^m(\O, d\omega)}|A_k|^{1-\frac 1m-\frac{1}{p^*_s}}_{d\omega}.
$$
Let $h>k$, since $A_h\subset A_k$, there results that
$$
(h-k) |A_h|^{\frac{p-1}{p^*_s}}_{d\omega} \leq \| (|f||x|^{p^*_s\b})\|_{L^m(\O, d\omega)}|A_k|^{1-\frac 1m-\frac{1}{p^*_s}}_{d\omega}.
$$
Hence
$$
|A_h|_{d\omega}\leq \frac{C \| (|f||x|^{p^*_s\b})\|^{\frac{p^*_s}{p-1}}_{L^m(\O, d\omega)}|A_k|^{\frac{p^*_s}{p-1}(1-\frac 1m-\frac{1}{p^*_s})}_{d\omega}
}{(h-k)^{\frac{p^*_s}{p-1}}}\,.
$$
We set $\Phi(k)=|A_h|_{d\omega}$, then
$$
\Phi(h)\le \frac{C \Phi^{\frac{p^*_s}{p-1}(1-\frac 1m-\frac{1}{p^*_s})}(k)
}{(h-k)^{\frac{p^*_s}{p-1}}}\,.
$$
Since $m>\frac{N}{ps}$, then $\frac{p^*_s}{p-1}(1-\frac 1m-\frac{1}{p^*_s})>1$. By the classical result of Stampacchia, see \cite{St}, we get the existence of
$k_0>0$ such that $\Phi(h)=0$ for all $h\ge k_0$, hence we conclude.
\end{pf}

{\bf Proof of Theorem \ref{exisqp}.}

We follow closely the argument used in the proof of Theorem \ref{exisqq}.

Define $u_n$ to be the unique solution to the approximated problem
\begin{equation}\label{appr11}
\left\{
\begin{array}{rcll}
(-\D)^s_{p,\beta} u_n & = & \l u^{p-1}_n+g_n & \mbox{ in }\O,\\
u_n & \ge & 0 & \mbox{  in   } \O,\\
u_n & = & 0 & \mbox {  in  }\ren\backslash \O.
\end{array}
\right.
\end{equation}
Notice that the existence of $u_n$ can be obtained as a critical point of the functional
$$
J(u_n)=\frac{1}{2p}\dyle\iint_{D_{\O}} \,|u(x)-u(y)|^{p}d\nu-\frac{\l}{p}\io |u|^{p}dx-\io g_n\: u\:dx.
$$
However the uniqueness follows using the comparison result in Lemma \ref{compa}. It is clear that using the same comparison principle we obtain that $u_n\le u_{n+1}$.

We claim that $\{u^{p-1}_n\}_n$ is uniformly bounded in $L^1(\O)$. We argue by contradiction. Assume that $C_n\equiv ||u^{p-1}_n||_{L^1(\O)}\to \infty$ as $n\to \infty$. We set $v_n=\frac{u_n}{C^{\frac{1}{p-1}}_n}$, then $||v^{p-1}_n||_{L^1(\O)}=1$ and $v_n$ solves the problem
\begin{equation}\label{appr13}
\left\{
\begin{array}{rcll}
(-\D)^s_{p,\beta} v_n & = & \l v^{p-1}_n+\dfrac{g_n}{C_n} &\mbox{ in }\O,\\
v_n & \ge & 0 & \mbox{  in   } \O,\\
v_n & = & 0 & \mbox {  in  }\ren\backslash \O.
\end{array}
\right.
\end{equation}
We set $G_n\equiv v^{p-1}_n+\frac{g_n}{C_n}$, then $||G_n||_{L^1(\O)}\le C$. Taking in consideration the results of Lemmas \ref{one} and \ref{two}
we get the existence of a measurable function $v$ such that $T_k(v)\in W^{s,p}_{\beta, 0}(\Omega)$, $v^{p-1}\in
L^\s(\O,|x|^{-2\beta} dx)$ for all $\s<\dfrac{N}{N-ps}$ and $T_k(v_n)\rightharpoonup T_k(v)$ weakly in $W^{s,p}_{\beta, 0}(\Omega)$.

Since $\s>1$, then using Vitali's Lemma we can prove that $v^{p-1}_n\to v^{p-1}$ strongly in $L^1(\O)$. Thus $||v^{p-1}||_{L^1(\O)}=1$. It is clear that $G_n\to \l v^{p-1}$ strongly in $L^1(\Omega)$. Thus $v$ solves

\begin{equation}\label{fina}
\left\{
\begin{array}{rcll}
(-\D)^s_{p,\beta} v & =& \l v^{p-1} & \mbox{ in }\O,\\
v & \gneq & 0 &  \mbox{  in   } \O,\\
v & = & 0 & \mbox {  in  }\ren\backslash \O.
\end{array}
\right.
\end{equation}
We claim that $v\in L^\infty(\Omega)$. From the previous discussion we know that $v^{p-1}\in
L^\s(\O,|x|^{-2\beta} dx)$ for all $\s<\dfrac{N}{N-ps}$. Thus setting $a_1=\dfrac{(p-1)N}{N-ps}-(p-1)-\epsilon$, with $\e$ very small,
and using an approximation argument, we can take $v^{a_1}$ as a test function in \eqref{fina} to conclude that
$$
\iint_{D_{\O}} \,|v(x)-v(y)|^{p-2}(v(x)-v(y))(v^{a_1}(x)-v^{a_1}(y))d\nu \le C.
$$
Hence using inequality \eqref{alge3}, it follows that
$$
\iint_{D_{\O}} \,|v^{\frac{a_1+p-1}{p}}(x)-v^{\frac{a_1+p-1}{p}}(y)|^pd\nu \le C.
$$
Using the Sobolev inequality in Lemma \ref{Sobolev}, we reach that
$$
\io \dfrac{|v(x)|^{(a_1+p-1)\frac{p^*_s}{p}}}{|x|^{p_{s}^{*}\beta}}dx<\infty.
$$
Now, we set
$a_2=(a_1+p-1)\frac{p^*_s}{p}-(p-1)$, then using $v^{a_2}$ as a test function in \eqref{fina} and following the same argument as above we conclude that
$$
\io\dfrac{|v(x)|^{(a_2+p-1)\frac{p^*_s}{p}}}{|x|^{p_{s}^{*}\beta}}dx<\infty.
$$
Consider now the sequence $a_{n+1}=(a_n+p-1)\frac{p^*_s}{p}-(p-1)$. It is clear that $a_n\uparrow\infty$ and by an induction argument we can prove that $\io v^{a_n}dx <\infty$ for all $n$. Thus using Theorem \ref{bound}, we conclude that $v$ is an energy solution to problem \eqref{fina} and that
$v\in L^\infty(\O)$. Now using $v$ as a test function in \eqref{fina}, taking in consideration that $\l<\l_1$, it follows that $||v||_{W^{s,p}_{\beta, 0}(\Omega)}=0$, a contradiction with the fact that $||v^{p-1}||_{L^1(\O)}=1$. Hence the claim follows. Now the rest of the proof follows exactly the same argument as in the proof of Theorem \ref{exisqq}. \cqd

In the case where $q>p-1$, we need to assume additional conditions on $g$. More precisely, if $g\in L^1(\O)$, we define $w$ to be
the unique positive solution to problem
\begin{equation}\label{eq:arg}
\left\{
\begin{array}{rcll}
(-\D)^s_{p,\beta} w & = & g & \mbox{ in }\O,\\
w & = & 0 & \mbox {  in  }\ren\backslash \O.
\end{array}
\right.
\end{equation}
We are able to prove the following result.
\begin{Theorem}\label{th:bbc}
Assume that $g\in L^1(\O)$ verifies $w^q(x)\le g(x)$ a.e. in $\Omega$, then there exists a positive constant $\bar{\l}$ such for all $\l<\bar{\l}$, the problem
(\ref{eq:exis2}) has a minimal entropy positive solution.
\end{Theorem}
\begin{pf} Recall that by the results of Lemmas \ref{one} and \ref{two}, we know that
$w^{p-1}\in
L^\s(\O,|x|^{-2\beta} dx)$ for all $\s<\dfrac{N}{N-ps}$ and $T_k(w)\in W^{s,p}_{\beta, 0}(\Omega)$.

Let $v$ be the minimal solution to the problem
\begin{equation}\label{vv1}
\left\{
\begin{array}{rcll}
(-\D)^s_{p,\beta} v &= & g+w^q & \mbox{ in }\O,\\
v & = & 0 & \mbox {  in  }\ren\backslash \O.
\end{array}
\right.
\end{equation}
It is not difficult to show that $ v\le
2^{p-1}w$, hence using the hypothesis on $g$, it follows that
$$ (-\D)^s_{p,\beta}v=g+w^q\ge g+
2^{\frac{q}{1-p}}v^q. $$ Then $v$ is a
supersolution to (\ref{eq:exis2}) for $\l\le \bar{\l}=2^{\frac{q}{1-p}}$. Fixed $\l$ as above and define the sequence $\{u_n\}_n$ by $u_0=0$, $u_{n+1}$ is the unique solution to the following problem
\begin{equation}\label{eq:arg1}
\left\{
\begin{array}{c}
(-\D)^s_{p,\beta} u_{n+1}=u^q_n+g_{n+1}\mbox{ in }\O,\\
u_{n+1}=0\mbox {  in  }\ren\backslash \O.
\end{array}
\right.
\end{equation}
By an induction argument we can prove that $u_n\le v$ for all $n$ and that the sequence $\{u_n\}_n$ is increasing in $n$.
Thus $\{u^q_n+g_n\}_n$ is increasing and bounded in $L^1(\O)$. Now, using the same compactness argument as in the proof of Theorems \ref{exisqq} and \ref{exisqp} we get the existence result.
\end{pf}

\section{Appendix}\label{sec:s04}
\subsection{Harnack Inequality }

This section is devoted to prove a weak version of the Harnack Inequality for positive supersolution to problem \eqref{eq:def}. Let begin by the next definition.
\begin{Definition}\label{super00}
Let $v\in W^{s,p}_{\b,loc}(\Omega)$, we say that $v$ is
supersolution to problem \eqref{eq:def} if for all
$\O_1\subset\subset \O$, we have
\begin{equation}\label{super}
\iint_{{D_{\Omega_1}}}|v(x)-v(y)|^{p-2}(v(x) - v(y))(\varphi(x)- \varphi(y))d\nu\ge
\dint_{\Omega_1}\,f\varphi\,dx,\end{equation} for any nonnegative
$ \varphi \in W^{s,p}_{\b, 0}(\Omega_1)$.
\end{Definition}
The main result of the appendix is the next version of the weak Harnack inequality.
\begin{Theorem}{\it(Weak Harnack inequality)}\label{harnack}\\
Assume that $f\ge 0$ and let $v\in W^{s,p}_\b (\ren)$ be  a
supersolution to \eqref{eq:def} with $v\gneqq 0$ in $\ren $. Then
for any $q<\frac{N(p-1)}{N-ps}$, we have
\begin{equation}\label{main}
\Big(\int_{B_r}v^q|x|^{-2\beta} dx\Big)^{\frac 1q}\le C
\inf_{B_{\frac{3r}{2}}}v.
\end{equation}
\end{Theorem}
Contrary to the local case where Moser type iteration is used to get the Harnack inequality, see \cite{Cs}, in this
we will use a different approach.

For $\beta=0$, the result was obtained in \cite{CKP} where a general version of The Harnack inequality is proved including for sign-changing solutions.

The case $\beta>0, p=2$ and positive datum was obtained in \cite{AMPP2}. Here we combine both arguments to prove Theorem \ref{harnack}.

Recall that $d\mu(x)\equiv \dfrac{dx}{|x|^{2\beta}}$ and $d\nu\equiv \dfrac{dxdy}{|x-y|^{N+ps}|x|^\beta|y|^\beta}$. Notice that we have just to consider the case where $B_r(x_0)=B_r(0)$.

For simplicity of  typing, we will write $B_r$ in place of $B_r(0)$. We will use systematically the next  weighted version of the
Poincar\'{e}-Wirtinger inequality. We refer to the Appendix of \cite{AMPP} for the proof.
\begin{Theorem}\label{PW}
Let $w\in W^{s,p}_\b(B_2)$ and assume that $\psi$ is a radial
decreasing function such that $\text{Supp}\:\psi\subset B_1$ and
$0\lneqq \psi\le 1$. Define
$$
W_\psi=\dfrac{\int_{B_{1}}w(x)\psi(x)d\mu}{\int_{B_1}\psi(x)d\mu},
$$
then $$ \int_{B_1}|w(x)-W_\psi|^p\psi(x)d\mu\le
C\int_{B_1} \int_{B_1}|w(x)-w(y)|^p\min\{\psi(x), \psi(y)\}d\nu.
$$
\end{Theorem}

Let us begin by proving the next lemma.
\begin{Lemma}\label{lema1}
Assume that $v\in W^{s,p}_\b(\ren)$ with $v\gvertneqq 0$, is a
supersolution to \eqref{eq:def}. Let $k>0$ and suppose that for
some $\sigma\in (0,1]$, we have
\begin{equation}\label{elli1}
|B_r\cap \{v\ge k\}|_{d\mu}\ge \sigma|B_r|_{d\mu}
\end{equation}
with $0<r<\frac{R}{16}$. Then there exists a positive constant
$C=C(N,s)$ such that
\begin{equation}\label{elli2}
|B_{6r}\cap \{v\le 2\delta k\}|_{d\mu}\le
\frac{C}{\sigma\log(\frac{1}{2\delta})}|B_{6r}|_{d\mu}
\end{equation}
for all $\delta\in (0,\frac 14)$.
\end{Lemma}
\begin{pf} Without loss of generality we can assume that $v>0$ in
$B_R$, (if not we can deal with $v+\e$ and we let $\e\to 0$.) Let
$\psi\in \mathcal{C}^\infty_0(B_R)$ be such that $0\le \psi\le 1$,
$\text{supp}\: \psi\subset B_{7r}$, $\psi=1$ in $B_{6r}$ and
$|\nabla \psi|\le \frac{C}{r}$.

Putting $\varphi=\psi^{p}v^{1-p}$ as a test function in \eqref{super}, it
follows that
$$
\int_{\mathbb{R}^{N}}\int_{\mathbb{R}^{N}}
|v(x)-v(y)|^{p-2}(v(x)-v(y))(\psi^{p}(x)v^{1-p}(x)-\psi^{p}(y)v^{1-p}(y)) d\nu\ge 0.
$$
Thus
\begin{eqnarray*}
& 0\dyle \le \int_{B_{8r}}\int_{B_{8r}}
|v(x)-v(y)|^{p-2}(v(x)-v(y))(\frac{\psi^{p}(x)}{v(x)^{p-1}}-\frac{\psi^{p}(y)}{v(y)^{p-1}})
d\nu
\end{eqnarray*}
$$+ 2\int_{\ren \setminus
B_{8r}}\int_{B_{8r}}|v(x)-v(y)|^{p-2}(v(x)-v(y))\frac{\psi^{p}(x)}{v(x)^{p-1}}d\nu.$$

It is not difficult to show that
$$
\dyle \int_{\ren \setminus
B_{8r}}\int_{B_{8r}}|v(x)-v(y)|^{p-2}(v(x)-v(y))\frac{\psi^{p}(x)}{v(x)^{p-1}}d\nu\le
\dyle\int_{\ren \setminus B_{8r}}\int_{B_{8r}}\psi^{p}(x)d\nu.
$$
Since
$$
\begin{array}{lll}
\dyle\int_{\ren \setminus B_{8r}}\int_{B_{8r}}\psi^{p}(x)d\nu & = & \dyle  \dint_{B_{7r}}
\frac{\psi^p(x)}{|x|^{\beta}}\int_{8r}^{\infty}
\dfrac{\rho^{N-\beta-1}d\rho}{|x|^{N+ps}}
\Big(\dint_{{\mathbb{S}}^{N-1}}\dfrac{dy'}{|\frac{\rho}{|x|}y'-x'|^{N+ps}}\Big)dx,
\end{array}
$$
setting $\tau=\dfrac{\rho}{|x|}$, it follows that
$$
\dyle\int_{\ren \setminus B_{8r}}\int_{B_{8r}}\psi^{p}(x)d\nu=\dint_{B_{7r}}\frac{\psi^{p}(x)dx}{|x|^{2\beta+ps}}\int_{\frac
87}^{\infty} \tau^{N-\beta-1}D(\tau)d\tau,
$$
where
\begin{equation*}
D(\tau)=2\frac{\pi^{\frac{N-1}{2}}}{\beta(\frac{N-1}{2})}\int_0^\pi
\frac{\sin^{N-2}(\theta)}{(1-2\sigma \cos
(\theta)+\tau^2)^{\frac{N+ps}{2}}}d\theta.
\end{equation*}
Taking in consideration the behavior of $D$ near $0, 1$ and
$\infty$, we obtain that
$$
\int_{\frac 87}^{\infty} \tau^{N-\beta-1}D(\tau)d\tau\le C.
$$
Therefore we conclude that
$$
\int_{\ren \setminus
B_{8r}}\int_{B_{8r}}|v(x)-v(y)|^{p-2}(v(x)-v(y))\frac{\psi^{p}(x)}{v(x)^{p-1}}d\nu\le
Cr^{N-ps-2\beta}.
$$
Notice that from \cite{CKP1}, we know that
\begin{equation*}
\begin{split}
&|v(x)-v(y)|^{p-2}(v(x)-v(y))(\frac{\psi^{p}(x)}{v(x)^{p-1}}-\frac{\psi^{p}(y)}{v(y)^{p-1}})\\
&\le
-C_1|\log(v(x))-\log(v(y))|^p \psi(y)^p+C_2(\psi(x)- \psi(y))^p.
\end{split}
\end{equation*}
Thus
$$
\begin{array}{lll}
&\dyle \int_{B_{8r}}\int_{B_{8r}}
|v(x)-v(y)|^{p-2}(v(x)-v(y))(\frac{\psi^{p}(x)}{v(x)^{p-1}}-\frac{\psi^{p}(y)}{v(y)^{p-1}})
d\nu\\
&&\\&=\dyle \int_{B_{6r}}\int_{B_{6r}}
|v(x)-v(y)|^{p-2}(v(x)-v(y))(\frac{1}{v(x)^{p-1}}-\frac{1}{v(y)^{p-1}}) d\nu\\ &&\\ &+ \dyle
\iint_{B_{8r}\times B_{8r} \setminus B_{6r}\times
B_{6r}}|v(x)-v(y)|^{p-2}(v(x)-v(y))(\frac{\psi^{p}(x)}{v(x)^{p-1}}-\frac{\psi^{p}(y)}{v(y)^{p-1}})d\nu\\
&&\\
&\dyle \le \int_{B_{6r}}\int_{B_{6r}}
|v(x)-v(y)|^{p-2}(v(x)-v(y))(\frac{1}{v(x)^{p-1}}-\frac{1}{v(y)^{p-1}}) d\nu+Cr^{N-ps-2\beta}.
\end{array}
$$
Hence combining the above estimates it follows that
\begin{equation}\label{elli3}
\int_{B_{6r}}\int_{B_{6r}} |\log(v(x))-\log(v(y))|^pd\nu\le
Cr^{N-ps-2\beta}.
\end{equation}
We set $w(x)=\min\{\log(\frac{1}{2\delta}),
\log(\frac{k}{v})\}_+$, then using \eqref{elli3}, there results
that
\begin{equation}\label{elli4}
\int_{B_{6r}}\int_{B_{6r}} |w(x)-w(y)|^pd\nu\le Cr^{N-ps-2\beta}.
\end{equation}
Define $$\langle
w\rangle_{B_{6r}}=\dfrac{1}{|B_{6r}|_{d\mu}}\int_{B_{6r}}w(x)d\mu,$$
then using H\"older inequality and the Poincar\'{e}-Wirtinger inequality,
$$
\int_{B_{6r}}|w(x)-\langle w\rangle_{B_{6r}}|d\mu\le
C|B_{6r}|_{d\mu}.
$$
Notice that $\{x\in\Omega/\,w(x)=0\}=\{x\in\Omega/ v(x)\ge k\}$, then from \eqref{elli1} we have
\begin{equation}\label{elli111}
|B_{6r}\cap \{v\ge k\}|_{d\mu}\le
\dfrac{\sigma}{6^{N-2\beta}}|B_{6r}|_{d\mu}.
\end{equation}
It is clear that
$$
B_{6r}\cap \{v\ge 2\delta k\}=B_{6r}\cap
\{w=\log(\frac{1}{2\delta})k\},
$$
then using the fact that
$$
|B_{6r}\cap \{w=\log(\frac{1}{2\delta})k\}|_{d\mu}\le
\dfrac{6^{N-2\beta}}{\sigma
\log(\frac{1}{2\delta})}\int_{B_{6r}}|w(x)-\langle
w\rangle_{B_{6r}}|d\mu,
$$
we get the desired result.
 \end{pf}

As a consequence we have the next estimate on $\inf\limits_{B_{4r}} v$.
\begin{Lemma}\label{lema2}
Assume that the hypotheses of Lemma \ref{lema1} are satisfied,
then there exists $\delta \in (0,\frac 12)$ such that
\begin{equation}\label{estim2} \inf_{B_{4r}} v\ge \delta k.
\end{equation}
\end{Lemma}
\begin{pf} We set $w=(l-v)_-$ where $l\in (\delta k,2\delta k)$ and let
$\psi\in \mathcal{C}^\infty_0(B_\rho)$ with $r\le \rho<6r$.

Putting  $\varphi=w\psi^p$ as a test function in \eqref{super} and following
the same computation as in the previous lemma, we reach that
\begin{equation*}
\begin{split}
&\int_{B_{\rho}}\int_{B_{\rho}}
|v(x)-v(y)|^{p-2}(v(x)-v(y))(w(x)\psi^{p}(x)-w(y)\psi^{p}(y))
d\nu\\
&\le -c \int_{B_{\rho}}\int_{B_\rho} |w(x)\psi(x)-w(y)\psi(y)|^p
d\nu\\
&+c \int_{B_{\rho}}\int_{B_\rho}((max\{w(x),w(y)\}))^p |\psi(x)-\psi(y)|^p
d\nu.
\end{split}
\end{equation*}
Thus, combining the above results, we get
\begin{equation}\label{last}
\begin{array}{lll}
&\dyle \int_{B_{\rho}}\int_{B_\rho} |w(x)\psi(x)-w(y)\psi(y)|^p
d\nu \le \\& C_1\dyle
\int_{B_{\rho}}\int_{B_\rho}((max\{w(x),w(y)\}))^p |\psi(x)-\psi(y)|^pd\nu +
l^p|B_\rho\cap \{v<l\}|_{d\mu}\times \sup_{\{x\in
\text{supp}\,\,\rho\}}\int_{\ren\backslash
B_{\rho}}\dfrac{dy}{|x-y|^{N+ps}}.
\end{array}
\end{equation}
We define now the sequences $\{l_j\}_{j\in\mathbb{N}}$,
$\{\rho_j\}_{j\in\mathbb{N}}$ and
$\{\bar{\rho_j}\}_{j\in\mathbb{N}}$ by setting
$$
l_j=\delta k+2^{-j-1}\delta k,\:\rho_j=4r+2^{1-j}r,\:
\bar{\rho}_j=\frac{\rho_j+\rho_{j+1}}{2}.
$$
Using the Sobolev inequality stated in Theorem \ref{Sobolev},
we obtain that
$$
C(N,s,\beta)\Big(\dint\limits_{B_j}
\dfrac{|w_j\psi_j(x)|^{p^*_{s}}}{|x|^{\beta
p^*_s}}\,dx\Big)^{\frac{p}{p^*_{s}}}\le
\dint_{B_j}\dint_{B_j}|w_j(x)\psi_j(x)-w_j(y)\psi_j(y)|^pd\nu.
$$
Therefore, using the fact that $w_j\psi_j\ge (l_j-l_{j+1})$ in
$B_{j+1}\cap \{v<l_{j+1}\}$ and taking in consideration that
$|x|^{-p^*_s\beta}\ge \bar{C}r^{-(p^*_s-2)\beta}|x|^{-2\beta}$ in $B_j$
with $\bar{C}$ is independent of $j$, it follows that
$$
\Big(\dint\limits_{B_j} \dfrac{|w_j\phi(x)|^{p^*_{s}}}{|x|^{\beta
p^*_s}}\,dx\Big)^{\frac{p}{p^*_{s}}}\ge \frac{C}{r^{(p^*_s-2)\beta
}}(l_j-l_{j+1})^p|B_{j+1} \cap \{v<j+1\}|^{\frac{p}{p^*_s}}_{d\mu}.
$$
Hence we conclude that
$$
(l_j-l_{j+1})^p\Big(\dfrac{|B_{j+1} \cap
\{v<j+1\}|_{d\mu}}{|B_{j+1}|_{d\mu}}\Big)^{\frac{p}{p^*_s}}\le
C(N,s)r^{-(N-ps-2\beta)}\dint_{B_j}\dint_{B_j}|w_j(x)\phi(x)-w_j(y)\phi(y)|^pd\nu.
$$
By application of \eqref{last} for $w_j$, we get
\begin{equation}\label{last1}
\begin{split}
& (l_j-l_{j+1})^p\Big(\dfrac{|B_{j+1} \cap
\{v<j+1\}|_{d\mu}}{|B_{j+1}|_{d\mu}}\Big)^{\frac{p}{p^*_s}}\\
&\le
\dfrac{C(N,s)}{r^{(N-ps-2\beta)}}\Big(C_1\dyle
\int_{B_{\rho}}\int_{B_\rho}((max\{w_j(x),w_j(y)\}))^p |\psi_j(x)-\psi_j(y)|^pd\nu \\
&+l^p_j|B_j\cap \{v<l_j\}|_{d\mu}\times \sup_{\{x\in
\text{supp}\,\,\rho_j\}}\int_{\ren\backslash
B_{\rho_j}}\dfrac{dy}{|x-y|^{N+ps}}\Big).
\end{split}
\end{equation}
We have
$$
\begin{array}{lll}
\dyle \int_{B_{j}}\int_{B_j}w^p_j|\psi_j(x)-\psi_j(y)|^pd\nu &\le
& \dyle l^p\int_{B_{j}\cap
\{v<l_j\}}\dfrac{dx}{|x|^\beta}\int_{B_j}\dfrac{|x-y|^{p-ps}}{|x-y|^N}\dfrac{dy}{|y|^\beta}\\
&\le & \dyle Cr^{-ps}\int_{B_{j}\cap
\{v<l_j\}}\dfrac{dx}{|x|^{2\beta}}.
\end{array}
$$
Thus
$$
\dyle \int_{B_{j}}\int_{B_j}w^p_j|\psi_j(x)-\psi_j(y)|^pd\nu\le
Cr^{-ps}|B_{j}\cap \{v<l_j\}|_{d\mu}.
$$
Now, estimating the term $\sup_{\{x\in
\text{supp}\,\,\rho_j\}}\int_{\ren\backslash
B_{\rho_j}}\dfrac{dy}{|x-y|^{N+ps}}$ as in \cite{CKP} we obtain
that
$$
(l_j-l_{j+1})^p\Big(\dfrac{|B_{j+1} \cap
\{v<j+1\}|_{d\mu}}{|B_{j+1}|_{d\mu}}\Big)^{\frac{p}{p^*_s}}\le\\
\dfrac{C(N,s)}{r^{(N-ps-2\beta)}}r^{-ps}|B_{j}\cap
\{v<l_j\}|_{d\mu}\le \tilde{C}\dfrac{|B_{j} \cap
\{v<j\}|_{d\mu}}{|B_{j}|_{d\mu}}
$$
where $\tilde{C}\le C(R,N,s)2^{N+s+2}$.

Define $A_j=\dfrac{|B_{j} \cap \{v<j\}|_{d\mu}}{|B_{j}|_{d\mu}}$
and following the same arguments as in \cite{CKP}, we get the desired result.
\end{pf}

Now, we need to obtain a kind of \emph{reverse H$\ddot{o}$lder inequality }
for $v$. More precisely we have the following result.
\begin{Lemma}\label{dos}
Suppose that $v$ is a supersolution to \eqref{eq:def}, then for all
$0<\alpha_1<\alpha_2<\frac{N(p-1)}{N-ps}$, we have
\begin{equation}\label{est3}
\Big(\dfrac{1}{|B_{r}|_{d\mu}}\dint_{B_{r}}v^{\alpha_2}\,d\mu\Big)^{\frac{1}{\alpha_2}}
\le C \Big(\dfrac{1}{|B_{\frac 32
r}|_{d\mu}}\dint_{B_{\frac32
r}}v^{\alpha_1}\,d\mu\Big)^{\frac{1}{\alpha_1}}.
\end{equation}
\end{Lemma}
\begin{pf}
Let $q\in (1,p)$ and $d>0$, we set $\tilde{v}=(v+d)$. Assume that
$\psi\in \mathcal{C}^\infty_0(B_R)$ is such that
$\text{Supp}\psi\subset B_{\tau r}, \psi=1$ in $B_{\tau' r}$ and
$|\nabla \psi|\le \frac{C}{(\tau-\tau')r}$ where $\frac 12 \le
\tau'<\tau<\frac 32$. Then using $\varphi=\tilde{v}^{1-q}\psi^p$ as a test
function in \eqref{super}, we obtain that
%\begin{eqnarray*}
%&
 $$ \int_{B_{r}}\int_{B_{r}}
|\tilde{v}(x)-\tilde{v}(y)|^{p-2}(\tilde{v}(x)-\tilde{v}(y))(\frac{\psi^{p}(x)}{\tilde{v}^{q-1}(x)}-\frac{\psi^{p}(y)}{\tilde{v}^{q-1}(y)})
d\nu$$
$$+ 2\int_{\ren \setminus
B_{r}}\int_{B_{r}}|\tilde{v}(x)-\tilde{v}(y)|^{p-2}(\tilde{v}(x)-\tilde{v}(y))\frac{\psi^{p}(x)}{\tilde{v}^{q-1}(x)}d\nu\ge 0.
$$
%\end{eqnarray*}
As in the proof of Lemma \ref{lema1}, we can prove that
$$
\int_{\ren \setminus
B_{r}}\int_{B_{r}}|\tilde{v}(x)-\tilde{v}(y)|^{p-2}(\tilde{v}(x)-\tilde{v}(y))\frac{\psi^{p}(x)}{\tilde{v}(x)^{q-1}}d\nu\le
C_1 \dint\limits_{B_r}\tilde{v}^{p-q}\psi^p\,d\mu \times
\sup_{\{x\in \text{Supp}\psi\}}\int_{\ren\backslash
B_{r}}\dfrac{dy}{|x-y|^{N+ps}}.
$$
In the same way, we get
$$
\begin{array}{lll}
&\dyle \int_{B_{r}}\int_{B_{r}}
|\tilde{v}(x)-\tilde{v}(y)|^{p-2}(\tilde{v}(x)-\tilde{v}(y))(\frac{\psi^{p}(x)}{\tilde{v}^{q-1}(x)}-\frac{\psi^{p}(y)}{\tilde{v}^{q-1}(y)})
d\nu\le\\
&\dyle -C_2\int_{B_{r}}\int_{B_{r}}
(\tilde{v}^{\frac{p-q}{p}}(x)-\tilde{v}^{\frac{p-q}{p}}(y))^p\psi^{p}(y)d\nu
+C_3\int_{B_{r}}\int_{B_{r}}
((\tilde{v}^{{p-q}}(x)+\tilde{v}^{p-q}(y))(\psi(x)-\psi(y))^p
d\nu.
\end{array}
$$
Since
$$
\int_{B_{r}}\int_{B_{r}}
((\tilde{v}^{{p-q}}(x)+\tilde{v}^{p-q}(y))|\psi(x)-\psi(y)|^p
d\nu\le \dfrac{Cr^{-ps}}{(\tau-\tau')^p}\dint\limits_{B_{\tau
r}}\tilde{v}^{p-q}\,d\mu,
$$
and
$$
\sup_{\{x\in \text{Supp}\psi\}}\int_{\ren\backslash
B_{r}}\dfrac{dy}{|x-y|^{N+ps}}\le Cr^{-ps},
$$
 then combining the above estimates we reach that
$$
\int_{B_{\rho}}\int_{B_\rho}((max\{w(x),w(y)\}))^p |\psi(x)-\psi(y)|^pd\nu
\le \dfrac{Cr^{-ps}}{(\tau-\tau')^p}\dint\limits_{B_{\tau
r}}\tilde{v}^{p-q}\,d\mu.
$$
Now using Sobolev inequality given in Proposition \ref{Sobolev}
and taking in consideration the previous estimates, there results
that
$$
\Big(\dfrac{1}{|B_{\tau' r}|_{d\mu}}\dint_{B_{\tau
r}}\tilde{v}^{\frac{(p-q)N}{N-ps}}\,d\mu\Big)^{\frac{N-ps}{N}} \le
\dfrac{C}{|B_{\tau r}|_{d\mu}(\tau-\tau')^p}\dint_{B_{\tau
r}}\tilde{v}^{p-q}\,d\mu.
$$
Now, letting $d\to 0$ and iterating the previous inequality, we
reach the desired result.
\end{pf}

The next Lemma will be the key in order to complete the proof of the weak Harnack inequality.
\begin{Lemma}\label{tres}
Assume that $v$ is a supersolution to \eqref{eq:def}, then there
exists $\eta\in (0,1)$ depending only on $N,s$ such that
\begin{equation}\label{est31}
\dyle\Big(\dfrac{1}{|B_r|_{d\mu}}\int_{B_r}v^{\eta}d\mu\Big)^{\frac{1}{\eta}}\le
C\inf_{B_r} v.
\end{equation}
\end{Lemma}

To prove Lemma \ref{tres} we need the next covering result, see
Lemma 4.1 in \cite{CKP}.
\begin{Lemma}\label{cover}
Assume that $E\subset B_r(x_0)$ is a measurable set. For
$\bar{\delta}\in (0,1)$, we define
$$
[E]_{\bar{\delta}}\equiv \bigcup_{\rho>0}\{B_{3\rho}(x)\cap
B_r(x_0), x\in B_r(x_0): |E\cap B_{3\rho}(x)|_{d\mu}>\bar{\delta}|
B_{\rho}(x)|_{d\mu}\}.
$$
Then, either
\begin{enumerate}
\item $
|[E]_{\bar{\delta}}|_{d\mu}\ge\frac{\tilde{C}}{\bar{\delta}}|E|_{d\mu}$,
or \item $ [E]_{\bar{\delta}}=B_r(x_0)$
\end{enumerate}
where $\tilde{C}$ depends only on $N$.
\end{Lemma}

{\bf Proof of Lemma \ref{tres}.}

Notice that, for any $\eta>0$,
\begin{equation}\label{rep}
\dyle\dfrac{1}{|B_r|_{d\mu}}\int_{B_r}v^{\eta}d\mu(x)=\eta
\int_0^\infty
t^{\eta-1}\dfrac{|B_r\cap\{v>t\}|_{d\mu}}{|B_r|_{d\mu}}dt.
\end{equation}
Then, for $t>0$ and $i\in \mathbb{N}$, we set $ A^i_t=\{x\in B_r:
v(x)>t\delta^i\}$ where $\delta $ is given by Lemma \ref{lema2}.
Notice that $A^{i-1}_t\subset A^i_t$.

Let $x\in B_r$ such that $B_{3\rho}(0)\cap B_r\subset
[A^{i-1}_t]_{\bar{\delta}}$, then
$$
|A^{i-1}_t\cap
B_{3\rho}(x)|_{d\mu}>\bar{\delta}|B_\rho|_{d\mu}=\frac{\bar{\delta}}{3^{N-2\beta}}|B_{3\rho}|_{d\mu}.
$$
Hence using Lemma \ref{lema2}, we reach that
$$
v(x)>\delta (t\delta^{i-1})=t\delta^i  \mbox{  for all }x\in B_r.
$$
Thus $[A^{i-1}_t]_{\bar{\delta}}\subset A^{i}_t$. Therefore, using
the alternative result in Lemma \ref{cover}, we obtain that,
either $A^{i}_t=B_r$ or $|A^{i}_t|_{d\mu}\ge
\frac{\tilde{C}}{\delta}|A^{i-1}_t|_{d\mu}$.

Thus, if for some $m\in \mathbb{N}$, we have
\begin{equation}\label{nesr}
|A^0_t|_{d\mu}>(\frac{\bar{\delta}}{\tilde{C}})^m|B_r|_{d\mu},
\end{equation}
then $|A^{m}_t|_{d\mu}= |B_r|_{d\mu}$. Hence $A^{i}_t=B_r$ and
then
$$
\inf_{B_r}v>t\delta^m.
$$
It is clear that \eqref{nesr} holds if
$m>\frac{1}{\log(\frac{\bar{\delta}}{\tilde{C}})}\log(\frac{|A^{0}_t|_{d\mu}}{|B_r|_{d\mu}})$.
Fixed $m$ as above and define
$\beta=\frac{\log(\frac{\bar{\delta}}{\tilde{C}})}{\log(\delta)}$,
it follows that
$$
\inf_{B_r}v>t\delta\Big(\frac{|A^{0}_t|_{d\mu}}{|B_r|_{d\mu}})^{\frac{1}{\beta}}.
$$
We set $\xi=\inf_{B_r}v$, then
$$
\dfrac{|B_r\cap\{v>t\}|_{d\mu}}{|B_r|_{d\mu}}=\frac{|A^{0}_t|_{d\mu}}{|B_r|_{d\mu}}\le
C\delta^{-\beta}t^{-\beta}\xi^\beta.
$$
Going back to \eqref{rep}, we have
$$
\dyle\dfrac{1}{|B_r|_{d\mu}}\int_{B_r}v^{\eta}d\mu(x)\le
\eta\int_0^at^{\eta-1}dt +\eta C\int_a^\infty
\delta^{-\beta}t^{-\beta}\xi^\beta dt.
$$
Choosing $a=\xi$ and $\eta=\frac{\beta}{2}$, we reach the desired
result.\cqd

We give now the proof of the weighted weak Harnack inequality.

{\bf Proof of Theorem \ref{harnack}.} Using Lemma \ref{tres} we
obtain that
$$
\dyle\Big(\dfrac{1}{|B_r|_{d\mu}}\int_{B_r}u^{\eta}d\mu(x)\Big)^{\frac{1}{\eta}}\le
C\inf_{B_r} u
$$
for some $\eta\in (0,1)$. Fixed $1\le q<\frac{N(p-1)}{N-ps}$, then by
Lemma \ref{dos} for $\alpha_1=\eta$ and $\alpha_2=q$, there
results that
\begin{equation}\label{est33}
\Big(\dfrac{1}{|B_{r}|_{d\mu}}\dint_{B_{
r}}u^{q}\,d\mu\Big)^{\frac{1}{q}} \le C \Big(\dfrac{1}{|B_{\frac
32 r}|_{d\mu}}\dint_{B_{\frac32
r}}u^{\eta}\,d\mu\Big)^{\frac{1}{\eta}}.
\end{equation}
Hence
$$
\Big(\dfrac{1}{|B_{r}|_{d\mu}}\dint\limits_{B_{
r}}u^{q}\,d\mu\Big)^{\frac{1}{q}} \le C\inf_{B_{\frac 32 r}} u
$$
and then we conclude. \cqd

\end{document}